\documentclass[11pt]{amsart}
\usepackage{amsmath,amssymb,latexsym,esint,cite,mathrsfs}
\usepackage{verbatim,wasysym}
\usepackage[left=2.6cm,right=2.6cm,top=2.9cm,bottom=2.9cm]{geometry}
\usepackage[T1]{fontenc}
\usepackage{lmodern}
\usepackage[colorlinks=true,urlcolor=blue, citecolor=red,linkcolor=blue,
linktocpage,pdfpagelabels, bookmarksnumbered,bookmarksopen]{hyperref}
\usepackage{tikz,enumitem,graphicx, subfig, microtype, color}
\usepackage{epic,eepic}
\usepackage[hyperpageref]{backref}
\usepackage[english]{babel}
\usepackage{xcolor}

\allowdisplaybreaks

\renewcommand{\epsilon}{\varepsilon}

\numberwithin{equation}{section}
\newtheorem{theorem}{Theorem}[section]
\newtheorem{proposition}[theorem]{Proposition}
\newtheorem{lemma}[theorem]{Lemma}
\newtheorem{remark}[theorem]{Remark}

\newtheorem{corollary}[theorem]{Corollary}
\newtheorem{definition}[theorem]{Definition}
\theoremstyle{definition}

\newcommand{\eps}{\varepsilon}

\renewcommand{\O}{{\mathcal O}}

\newcommand{\R}{\mathbb{R}}
\newcommand{\J}{J}
\newcommand{\weakto}{\rightharpoonup}

\begin{document}

\title{Nonlocal problems with critical Hardy nonlinearity}

\author[W.\ Chen]{Wenjing Chen}
\author[S.\ Mosconi]{Sunra Mosconi}
\author[M.\ Squassina]{Marco Squassina}

\address[W.\ Chen]{School of Mathematics and Statistics \newline\indent
Southwest University \newline\indent
Chongqing 400715, People's Republic of China.}
\email{wjchen@swu.edu.cn}

\address[S.\ Mosconi]{Dipartimento di Matematica e Informatica \newline\indent
	Universit\`a degli Studi di Catania \newline\indent
	Viale A. Doria 6, 95125 Catania, Italy}
\email{mosconi@dmi.unict.it}

\address[M.\ Squassina]{Dipartimento di Matematica e Fisica \newline\indent
	Universit\`a Cattolica del Sacro Cuore \newline\indent
	Via dei Musei 41, I-25121 Brescia, Italy}
\email{marco.squassina@unicatt.it}

\subjclass[2010]{35J20, 35J60,  47G20}
\keywords{Fractional $p$-Laplacian problems, critical exponents, Sobolev-Hardy inequality}
\thanks{W.\ Chen is supported by the National Natural Science Foundation of China No.\ 11501468 and
	by the Natural Science Foundation of Chongqing cstc2016jcyjA0323.\
	S. Mosconi and M.\ Squassina are members of the Gruppo
	Nazionale per l'Analisi Matematica, la Probabilit\`a e le loro Applicazioni}

\begin{abstract}
By means of variational methods we establish existence and multiplicity of solutions for a class of nonlinear nonlocal problems involving the fractional
$p$-Laplacian and a combined Sobolev and Hardy nonlinearity at subcritical and
critical growth.
\end{abstract}

\maketitle

\begin{center}
	\begin{minipage}{8.3cm}
		\small
		\tableofcontents
	\end{minipage}
\end{center}


\section{Introduction}\label{intro}

Let $\Omega\subset \R^N$ be a smooth open bounded set containing $0$.\
In this  work we study the existence and multiplicity of solutions to the following nonlocal problem driven by the fractional $p$-Laplacian operator
\begin{eqnarray}\label{frac1}
\left\{ \arraycolsep=1.5pt
   \begin{array}{lll}
(-\Delta_p)^s u =\lambda |u|^{r-2}u+\mu\frac{|u|^{q-2}u}{|x|^\alpha}\ \   &
{\rm in}\ \Omega, \\[2mm]
u= 0 \ \  & {\rm in}\ \R^N\setminus\Omega,
\end{array}
\right.
\end{eqnarray}
being $0<s<1$, $p>1$, $\lambda,\mu>0$, $0\leq \alpha\leq ps<N$,
$p\leq r\leq p^*,$ $p\leq q\leq p^*_\alpha$, where
\begin{equation}
\label{critexp}
p^*:=\frac{Np}{N-ps},\qquad
p^*_\alpha:=\frac{(N-\alpha)p}{N-ps},
\end{equation}
are the fractional critical {\em Sobolev} and  {\em Hardy--Sobolev} exponents respectively.
The fractional $(p, s)$-Laplacian operator $(-\Delta_p)^s$ is
the differential of the convex functional
\[
u\mapsto \frac{1}{p}[u]_{s,p}^p := \frac{1}{p}\int_{\R^{2N}} \frac{|u(x) - u(y)|^p}{|x - y|^{N+ps}}\, dx \, dy
\]
defined on the Banach space (with respect to the norm $[u]_{s,p}$ defined above)
\[
W^{s,p}_0(\Omega):=\{u\in L^1_{\rm loc}(\R^N): u\equiv 0 \ \text{in $\R^N\setminus\Omega$ and }\, [u]_{s,p}<+\infty\}.
\]
This definition is consistent, up to a normalization constant, with the linear
fractional Laplacian $(-\Delta)^s$ for the case $p=2$. The exponents in \eqref{critexp}
arise from the general fractional Hardy-Sobolev inequality 
\begin{equation}
\label{HaSo}
\left(\int_{\R^N} \frac{|u|^{p^*_\alpha}}{|x|^\alpha}\, dx\right)^{1/p^*_\alpha}\leq C(N, p, \alpha)[u]_{s, p}.
\end{equation}
The latter is a scale invariant inequality and as such is critical for the embedding
\[
W^{s,p}_0(\Omega)\hookrightarrow L^{q}\left(\Omega, \frac{dx}{|x|^\alpha}\right)
\]
in the sense that the latter is continuous for any $q\in [1, p^*_\alpha]$ but (as long as $0\in \Omega$, as we are assuming) is compact if and only if $q<p^*_\alpha$\footnote{See the proof of Lemma \ref{sob-func} for the {\em if} part. When $q=p^*_\alpha$ the scale invariance of the Hardy-Sobolev inequality implies the failure of compactness as long as $0\in \Omega$: given $\varphi\in C^{\infty}_c(\Omega)$, set for $\eps\in \ ]0, 1[$ $\varphi_\eps(x):=\varphi(x/\eps)$, and consider the family $\{\varphi_\eps/[\varphi_\eps]_{s,p}\}_\eps$.} 

We are interested in obtaining solutions with sign information of \eqref{frac1} with particular emphasis on the critical case $q=p^*_\alpha$. The term $|u|^{r-2}u$ in \eqref{frac1} will always be assumed subcritical, i.e. $r<p^*$. We will say that the equation is {\em subcritical} if the singular term  satisifes $q<p^*_\alpha$ as well, and {\em critical} if $q=p^*_\alpha$. Sometimes we will specify the nature of the criticality in the equation saying that \eqref{frac1} is 
\begin{itemize}
 \item
 {\em pure Hardy critical} if $\alpha=ps$, $q=p^*_\alpha=p$;
 \item
 {\em Hardy-Sobolev critical} if $\alpha\in \ ]0, ps[$, $q=p^*_\alpha$;
 \item
 {\em Sobolev critical} if $\alpha=0$, $q=p^*_\alpha=p^*$.
 \end{itemize}
 A similar analysis could in principle be performed in the case when $r=p^*$ and $q<p^*_\alpha$, i.e., when the singular term is subcritical while the unweighted term is critical, but we feel that doing this here would overcomplicate an already intricated framework.

\bigskip 
 
{\em -- Overview of the method}

Our approach consists in finding {\em ground state} solutions for \eqref{frac1}, i.e. to find solution which minimize the corresponding energy functional 
\[
J(u)=\frac{1}{p}[u]_{s,p}^p-\frac{\lambda}{r}\int_\Omega|u|^r\, dx-\frac{\mu}{q}\int_{\Omega}\frac{|u|^q}{|x|^\alpha}\, dx,
\]
among, respectively, the set of all positive or sign-changing solutions to \eqref{frac1}. In order to do so we define the Nehari, positive Nehari and sign-changing Nehari sets as
\[
\begin{split}
{\mathcal N}:&=\{u\in W^{s, p}_0(\Omega)\setminus\{0\}:  \langle J'(u), u\rangle =0\},\\
{\mathcal N}_+:&=\{u\in {\mathcal N}: u^-=0\},\\
{\mathcal N}_{\rm sc}:&=\{u\in {\mathcal N}: u^\pm\neq 0, \langle J'(u), u^\pm\rangle =0\},
\end{split}
\]
where $u^{\pm}$ denote the positive and negative part of $u$, respectively, and $\langle \ , \ \rangle:(W^{s, p}_0(\Omega))^*\times W^{s,p}_0(\Omega)\to \R$ is the duality pairing. Clearly, any positive weak solution of \eqref{frac1} lies in ${\mathcal N}_+$ and any sign-changing one belongs to ${\mathcal N}_{\rm sc}$.

All these Nehari-type sets are nonempty (but still quite far from being manifolds), and minimizing the energy functional on the latters actually provides weak solutions to \eqref{frac1} which are therefore called the positive and sign-changing {\em ground states} for \eqref{frac1}. Let us now describe the obstacles and differences with the classical case we will encounter in trying to solve these minimization problems.

\bigskip

{\em -- Positive solutions}

During all the paper we will focus on the coercive (with respect to the Nehari sets) case so that $\inf_{\mathcal N} J> 0$. Existence of positive ground states is by now standard and follows from a classical mountain pass procedure, once one can prove that the ground energy $\inf_{\mathcal N} J$ lies below a suitable compactness threshold. This threshold is usually called the energy of a {\em bubble}: it is given considering a function realizing the best constant in \eqref{HaSo}, rescaling it, and computing the limit energy as the corresponding "bubble" concentrates to a point. Unfortunately, the explicit form of the functions realizing \eqref{HaSo} (called {\em Aubin-Talenti functions}) is unknown in the fractional quasilinear case, preventing a direct estimate of the ground energy. In the local case $s=1$, these optimizers where explicitly obtained in \cite{Aubin, Talenti} when $\alpha=0$ and in \cite{gy} for general $\alpha\in \ ]0, p[$. In the fractional case $s\in \ ]0, 1[$, these are known only in the linear setting $p=2$, see \cite{Lieb}. When $p\neq 2$ the lack of an explicit form has been circumvented in \cite{mpsy} using an a-priori decay estimate for the Aubin-Talenti functions proved in \cite{brasco} and a suitable truncation technique via composition. Here, we modify the technique of \cite{mpsy} to cater with the general Hardy-Sobolev case $\alpha>0$, thanks to a similar decay estimates recently proved in \cite{mm}.

\bigskip

{\em -- Sign-changing solutions in the subcritical case}

In the sign-changing case minimization over ${\mathcal N}_{\rm sc}$ is less trivial, even in the subcritical setting. Indeed, it is not even clear whether ${\mathcal N}_{\rm sc}$ is empty or not. In the classical case $s=1$ one readily has that ${\mathcal N}_{\rm sc}\neq \emptyset$, since the locality of the energy gives
\[
{\mathcal N}_{\rm sc}=\{u\in {\mathcal N}: u^\pm\in {\mathcal N}\}
\]
and a classical scaling argument separately on $u^+$ and $u^-$ gives suitable $t^\pm$ such that $t^+u^++t^-u^-\in {\mathcal N}_{\rm sc}$. This is not the case in the nonlocal setting, since in general (see Remark \ref{remneh})
\[
\text{$u$ is a sign-changing solution to \eqref{frac1}}\quad \Rightarrow\quad \text{$u^+ \notin {\mathcal N}$ or $u^- \notin {\mathcal N}$}.
\]
This is basically due to the nonlocal interactions between $u^+$ and $u^-$ in the term $[u]_{s,p}$, given by 
\[
[u]_{s,p}^p-[u^+]_{s,p}^p-[u^-]_{s,p}^p,
\]
which is always strictly positive if $u^\pm \neq 0$. Nevertheless, we prove in Lemma \ref{le51} that actually ${\mathcal N}_{\rm sc}\neq \emptyset$, allowing to obtain sign-changing ground states via minimization over the latter, at least in the subcritical case.

\bigskip

{\em --Sign-changing solutions in the critical case}

 In the critical case, trying to solve $\inf_{{\mathcal N}_{\rm sc}} J$ through the direct method of Calculus of Variations faces the problem that $J$ fails to be weakly sequentially lower semicontinuous for large values of $\mu$. 

Instead, we perform a so-called {\em quasi-critical} approximation as in \cite{CSS, Ta}, considering the sign-changing ground states $u_\eps$ of \eqref{frac1} with $q=p^*_\alpha-\eps$, to obtain in the limit $\eps\downarrow 0$ a sign-changing ground state $u$ for the critical problem. In doing so, it is essential to estimate the asymptotic behavior of the sign-changing ground energies during the approximation. The core estimate (Lemma~\ref{energy-control}) says that the sign-changing ground energies stay {\em strictly} below the positive ground energy plus the energy of a bubble. Then, to prove that the limit is sign-changing, we employ a nonlocal version of the Concentration-Compactness principle recently proved in \cite{MS}. The latter shows that if the limit $u$ of the sign-changing ground states is, say, positive, then the negative part must concentrate in the limit, producing at least the energy of a bubble. Therefore the limit energy turns out to be at least the positive ground energy plus the one of a bubble, contradicting the previous asymptotic estimate.

\bigskip

{\em -- The Hardy critical case}

When $\alpha=ps$ and $q=p$, then problem \eqref{frac1} becomes the pure Hardy-critical equation
\[
(-\Delta_p)^su=\mu\frac{|u|^{p-2}u}{|x|^{ps}}+\lambda |u|^{r-2}u.
\]
On one hand there simply is no room in this case to perform the quasi-critical approximation since lowering the critical exponent drastically changes the geometry of the functional. On the other one, the Hardy-critical case of \eqref{frac1} lacks sufficient regularity estimates since one expects the solutions to be unbounded. In this case we are able to construct sign-changing solution proving the lower-semicontinuity of $J$ via Concentration--Compactess as in \cite{Montefusco}, at least when $\mu$ is below the optimal Hardy constant (coercive case). Notice that, for larger values of $\mu$, $J$ fails to be coercive over ${\mathcal N}$ but the geometry of the functional is richer and it is possible that sign-changing solutions can be found through linking arguments.

\bigskip

{\em -- Main results}

In order to state our results, we define for $\alpha\in [0, ps]$, the first weighted Dirichlet eigenvalue as
\begin{equation}
\label{lambdaalpha}
\lambda_{1, \alpha}:=\inf_{u\in W^{s,p}_0(\Omega)\setminus \{0\}}\frac{[u]_{s,p}^p}{\|u\, |x|^{-\alpha/p}\|_{L^p}^p}>0,\qquad \lambda_{1, 0}:=\lambda_1.
\end{equation}

\begin{theorem}[Subcritical case]
	\label{bothsub}
	Suppose that $0\leq \alpha< ps$ and 
 \[
q\in [p, p^*_\alpha[, \quad r\in [p, p^*[,\qquad \max\{q, r\}>p,\qquad 
\begin{cases}
\lambda>0,\  \mu>0&\text{if $\min\{q, r\}>p$},\\
0<\lambda<\lambda_1,\  \mu>0&\text{if $r=p$},\\
\lambda>0,\ 0<\mu<\lambda_{1,\alpha}&\text{if $q=p$}.
\end{cases}
\]
Then \eqref{frac1} has a positive 
and a sign-changing solution, both of minimal energy.
\end{theorem}

\begin{theorem}[Critical Hardy-Sobolev - positive solution]
	\label{theorem2}
Suppose that $0\leq \alpha< ps<N$, $r\in [p, p^*[$, $q=p^*_\alpha$, $\lambda, \mu>0$ and $0<\lambda<\lambda_1$ whenever $r=p$. Then \eqref{frac1} has a positive solution of minimal energy if 
\[
\begin{cases}
r\geq p&\text{if $N\geq p^2s$},\\
r>p^*-p'&\text{if $N<p^2s$}.
\end{cases}
\]
Here $p'=\frac{p}{p-1}$ denotes the dual exponent of $p$ for any $p>1$.
 \end{theorem}

\begin{theorem}[Critical Hardy-Sobolev - sign-changing solution]
	\label{sctheo}
Suppose that $0\leq \alpha< ps<N$, $r\in [p, p^*[$, $q=p^*_\alpha$, $\lambda, \mu>0$ and $0<\lambda<\lambda_1$ whenever $r=p$. Then \eqref{frac1} has a sign-changing solution of minimal energy if $r>p^*-1$.
\end{theorem}

\begin{theorem}[Pure Hardy critical]
\label{sctheoH}
Let $\alpha=ps<N$, $r\in \ ]p, p^*[$, $q=p^*_\alpha=p$, $\lambda, \mu>0$ and $\mu<\lambda_{1,ps}$. Then \eqref{frac1} has both a positive and a sign-changing solution of minimal energy.
\end{theorem}

{\em -- Discussion}
\begin{itemize}
\item
The assumption $\max\{r, q\}>p$ in the previous Theorems is made in order to avoid a fully homogeneous problem. Solutions of \eqref{frac1} with $r=q=p$ fall in the framework of weighted eigenvalue problems. As long as $\alpha<ps$ (subcritical case) these can be treated as in \cite{SW}. The only difficult case is the critical one  $\alpha=ps$, $p=q=r$, since the Hardy term fails to be compact. Even in the classical case $s=1$, this problem is quite delicate and has been treated in \cite{Tertikas}, \cite{Smets} through a weighted Concentration-Compactness alternative. 

\item
When $r=p$ and $q=p^*_\alpha>p$ (as in Theorems \ref{theorem2} and \ref{sctheo}), we are in the so called {\em Brezis-Nirenberg problems} framework mentioned before, where the critical nonlinearity and the spectrum of the principal part interact. Regarding positive solutions, observe that $N<p^2s$ implies that $p^*-p'>p$, so that we obtain positive solutions only for $0<\lambda<\lambda_1$ and $N\geq p^2s$. 
\item
When $r=p$ and  $\lambda\geq \lambda_1$, \eqref{frac1} has no constant-sign solutions by a well known argument (see e.g. \cite[Corollary A4]{LLM}) which we recall here. A Picone identity in the fractional setting has been obtained in \cite{bf} and reads as 
\[
[v]_{s,p}^p\geq \langle (-\Delta_p)^s u,\frac{v^p}{(u+\eps)^{p-1}}\rangle,\qquad \forall \eps>0, \quad u,v\geq 0.
\]
If $u$ is a non-negative, non-trivial solution to \eqref{frac1} and $v$ is the first eigenfunction for $(-\Delta_p)^s$, then
\[
\lambda_{1}\|v\|_{L^p}^p=[v]_{s,p}^p\geq \lambda\int_{\Omega}u^{p-1}\frac{v^p}{(u+\eps)^{p-1}}\, dx+\mu\int_{\Omega}\frac{u^{q-1}}{|x|^\alpha}\frac{v^p}{(u+\eps)^{p-1}}\, dx.
\]
Due to $\mu>0$, the second term on the right is uniformly (in $\eps>0$) bounded from below by a positive constant, and by dominated convergence on the first one we get $\lambda_{1}\|v\|_{L^p}^p>\lambda\|v\|_{L^p}^p$, forcing $\lambda<\lambda_1$.
 \item
Regarding the sing-changing result Theorem \ref{sctheo} in the case $r=p$, we obtain nodal solutions when $N>p^2s+ps$ and $\lambda<\lambda_1$. For $\lambda\geq \lambda_1$, the previous argument shows that actually {\em any} solution must be sign-changing, so that in a certain sense the problem is simpler. Solutions in this non-coercive case are obtained through linking arguments over cones, as exploited in \cite{mpsy} in the case $\alpha=0$ and in \cite{Yang} when $\alpha>0$, yielding sign-changing solutions for any $N>p^2s$ if $\lambda>\lambda_1$ does not belong to a suitable cohomological spectrum for $(-\Delta_p)^s$. The case when $\lambda$ belongs to the spectrum of $(-\Delta_p)^s$ is also discussed in these works, yielding a sign-changing solution whenever $\alpha=0$ and, e.g., $N\geq p^2s+ps$ (actually for even lower values of $N$).  
\item
We did not consider the doubly critical problem $r=p^*$, $q=p^*_\alpha$. In this case, the problem is much more difficult even in the classical case $s=1$, since its solvability heavily depends on the topology and geometry of the domain $\Omega$. Entire ground state solutions where found in the semilinear case in \cite{T} but, to the best of our knowledge, there is no explicit charactezation known in the case $p\neq 2$. For star-shaped domains with respect to the origin, a Pohozaev identity rules out existence of solutions when $p=2$ (\cite{T}). For $p\neq 2$, the lack of a unique continuation principle only ensures a similar result for constant-sign solutions, see \cite[Theorem 2.1]{gy}. For domains with nontrivial topology, existence can be granted for sufficiently small $\mu$, see \cite{HS}.
In the non-local, non-linear case the situation is more difficult due to the lack of a Pohozaev identity for $(-\Delta_p)^s$ when $p\neq 2$. The latter is available only in the case $p=2$ due to \cite{ROS}, leading in a standard way to the above mentioned non-existence result. Some existence results for contractible domains in the fractional semilinear case may be obtained following the ideas of \cite{MSS}.
\end{itemize}

{\em -- Comparison with some previous results}

The fractional Brezis-Nirenberg case $r=p$, $\alpha=0$, $q=p^*$ has been considered  in \cite{SV1, SV2} when $p=2$ and in \cite{mpsy} for any $p>1$; the paper \cite{Yang} deals with the quasilinear Brezis-Nirenberg problem with  Hardy-Sobolev nonlinearity, i.e. $r=p$ and $\alpha\in \ ]0, ps[$. In all these works, under various additional assumptions on the parameter $\lambda$ and its relation with a suitable spectrum of $(-\Delta_p)^s$, existence of a nontrivial solution is proved. Sign informations on the latter can in each case be obtained, either through its variational characterization (in the case $\lambda\in \ ]0, \lambda_1[$), or using the Picone identity argument described above when $\lambda\geq \lambda_1$. In particular, said solution turns out to be positive when $\lambda<\lambda_1$ and must be sign-changing when $\lambda\geq \lambda_1$. In this framework, the main novelty of this manuscript consists in finding sign-changing (least energy) solutions when $\lambda\in \ ]0, \lambda_1[$.\\
The local counterpart of \eqref{frac1}, is the quasi-linear problem
\[
-\Delta_p u =\lambda |u|^{r-2}u+\mu\frac{|u|^{q-2}u}{|x|^\alpha}\quad\text{in $\Omega$,}
\]
The Brezis-Nirenberg case $r=p$, $\alpha=0$, $q=p^*$ is well studied, see \cite{DGL} and the references therein. The general case was investigated in \cite{gy}, following the ideas of \cite{Ta}. However, unfortunately, there seems to be a gap in the proof when $p\neq 2$ (see e.g. the proof of Theorem 5.3, third display of p. 5720). Concerning the case $p=2$ and $0<s<1$, we refer the reader e.g.\ to \cite{wang,Yang-Yu}.

\medskip

{\em -- Structure of the paper}

\noindent
$\bullet$ In Section 2.1 we describe the relevant functional analytic setting. It is worth noting that for any $p>1$, the operator $(-\Delta_p)^s$ turns out to be sequentially continuous with respect to the weak topology both in $W^{s,p}_0(\Omega)$ and $W^{-s, p'}(\Omega)$; this is only true in the local case when $p=2$. Section 2.2 contains some basic properties of the Nehari sets described above. Section 2.3  recalls known decay properties for the Aubin-Talenti functions optimizing the Hardy-Sobolev inequality, while Section 2.4 contains technical estimates for suitable truncations of the latters. Section 2.5 collects quite classical results on the compactness of Palais-Smale sequences. 

\noindent
$\bullet$
Section 3.1 is devoted to the proof of Theorems \ref{bothsub} and \ref{theorem2}. In Section 3.2 we prove some uniform regularity estimates and obtain as a consequence the compactness of the positive ground states in the quasicritical approximation.

\noindent
$\bullet$
Section 4 concerns the problem of sign-changing solutions. In Section 4.1 we consider the subcritical case, obtaining the nodal ground states. Section 4.2 contains the core estimate for the sign-changing ground level in the subcritical approximation, which is then applied in Section 4.3 to prove Theorems \ref{sctheo} and \ref{sctheoH} through a Concentration-Compactness principle.

 \section{Preliminary results}
\label{sobhardyrt}
\subsection{Notations and functional spaces}
In the whole paper, we will assume that  $s  \in \ ]0, 1[$, $p>1$ and $0\leq \alpha\leq ps<N$. Given any $q\in \ ]1, +\infty[$ we let $q'=\frac{q}{q-1}$, and given $N, p, s, \alpha$ as before we set
\[
p^*_\alpha=\frac{(N-\alpha)p}{N-ps},\qquad p^*=p^*_0.
\]
Moreover, $\Omega\subseteq \R^N$ will be a open bounded set with  smooth (say, $C^2$) boundary. Given $E\subseteq\R^N$, $|E|$ will denote its Lebesgue measure, and $E^c=\R^N\setminus E$. All functions will be tacitly assumed to be Lebesgue measurable.

We introduce the fractional Sobolev space
\[
W^{s,p}_0(\Omega)=\left\{u\in L^1_{\rm loc}(\R^N): [u]_{s,p}^p:=\int_{\R^{2N}}\frac{|u(x)-u(y)|^p}{|x-y|^{N+ps}}\, dx\, dy<+\infty, \  \text{$u\equiv 0$ a.e. on $\Omega^c$}\right\}.
\]
and the homogeneous fractional Sobolev space
\[
D^{s,p}(\R^N)=\left\{u\in L^{p^*}(\R^N): [u]_{s,p}^p:=\int_{\R^{2N}}\frac{|u(x)-u(y)|^p}{|x-y|^{N+ps}}\, dx\, dy<+\infty\right\}\supset W^{s,p}_0(\Omega).
\]
For $p>1$, $W^{s,p}_0(\Omega)$ and $D^{s,p}(\R^N)$ are separable reflexive Banach space w.r.t. the norm $[\, \cdot\, ]_{s,p}$ and both can also be seen as the completion with respect to the norm $[\ ]_{s, p}$ of $C^\infty_c(\R^N)$ (see e.g. \cite[Theorem 2.1]{brasco}). The topological dual of $W^{s,p}_0(\Omega)$ will be denoted by $W^{-s, p'}(\Omega)$, with corresponding duality pairing $\langle \, \cdot\, , \, \cdot\, \rangle:W^{-s, p'}(\Omega)\times W^{s,p}_0(\Omega)\to \R$. Due to reflexivity, the weak and weak$^*$ convergence in $W^{-s, p'}(\Omega)$ coincide.
\medskip
\noindent
We recall here the fractional Hardy-Sobolev inequality.
\begin{lemma}[Hardy-Sobolev inequality]
	\label{sobhardy}
Assume that $0\leq \alpha\leq ps<N$. Then
there exists a positive constant $C$ such that
\begin{align*}
 \Big(\int_\Omega\frac{|u|^{p^*_\alpha}}{|x|^{\alpha}}dx\Big)^{1/p^*_\alpha}
\leq C \Big(\int_{\R^{2N}}\frac{|u(x)-u(y)|^p}{|x-y|^{N+ps}}\, dx\, dy\Big)^{1/p},
\qquad\text{for every $u\in W^{s,p}_0(\Omega)$.}
\end{align*}
\end{lemma}
\begin{proof}
It suffices to prove the inequality for $\Omega=\R^N$ and $u\in D^{s,p}(\R^N)$.  The latter is well known in the cases $\alpha=0$ and $\alpha=ps$, see \cite{frank} or \cite{MaSh}. The general case follows by H\"older's inequality since for $\theta=\frac{\alpha}{sp^*_\alpha}\in  [0, 1]$ it holds
\[
\|u|x|^{-\frac{\alpha}{p^*_\alpha}}\|_{L^{p^*_\alpha}(\R^N)}\leq \|u|x|^{-s}\|_{L^p(\R^N)}^{\theta}\|u\|_{L^{p^*}(\R^N)}^{1-\theta},
\]
giving the claimed inequality through the ones for $\alpha=0, ps$. 
\end{proof} 
In particular, $W^{s,p}_0(\Omega)$ embeds continuously into $L^q(\Omega, dx/|x|^\alpha)$ for all $\alpha\in [0, ps]$ and $q\in [1,  p^*_\alpha]$. Moreover, if $q\in [1, p^*_\alpha[$, the embedding is compact (see Lemma \ref{sob-func} below for a proof).  As a further consequence of the previous lemma we can define for any $\alpha\in [0, ps]$ the positive numbers 
\begin{equation}
\label{criticalfrachardyg}
S_\alpha  =\inf\left\{ \int_{\R^{2N}}\frac{|u(x)-u(y)|^p}{|x-y|^{N+ps}}\, dx\, dy:
\text{$u\in  W^{s,p}_0(\Omega)$ with $\int_{\Omega}\frac{|u|^{p^*_\alpha}}{|x|^{\alpha}}dx=1$}\right\}.
\end{equation}

Recalling \eqref{lambdaalpha}, it clearly holds $\lambda_{1,ps}=S_{ps}$.

\medskip
\noindent

The functional $u\mapsto \frac{1}{p}[u]_{s,p}^p$ is convex and $C^1$, so that at any $u\in W^{s,p}_{0}(\Omega)  $ its subdifferential, usually called {\em fractional $p\, s$-Laplacian}, is a uniquely defined element of $W^{-s, p'}(\Omega)$, which we will denote with $(-\Delta_p)^s u$. An explicit computation shows that
\begin{equation}
\label{splap}
\langle (-\Delta_p)^s u, \varphi\rangle=\int_{\R^{2N}}\frac{|u(x)-u(y)|^{p-2}(u(x)-u(y))(\varphi(x)-\varphi(y))}{|x-y|^{N+ps}}dx\, dy,\quad \forall \varphi\in W^{s,p}_0(\Omega).
\end{equation}

\begin{lemma}
\label{wtwcont}
The operator $(-\Delta_p)^s:W^{s,p}_0(\Omega)\to W^{-s, p'}(\Omega)$ is weak-to-weak continuous.
\end{lemma}

\begin{proof}
Let $u_n\weakto u$ in $W^{s,p}_0(\Omega)$. Since $u\mapsto \frac{1}{p}[u]_{s,p}^p$ is $C^1$, the sequence $(-\Delta_p)^s u_n$ is bounded in $W^{-s, p'}(\Omega)$, and thus weakly sequentially compact by reflexivity. By a standard sub-subsequence argument we can thus assume that $(-\Delta_p)^s u_n\to \Lambda\in W^{-s, p'}(\Omega)$ and are reduced to prove that $\Lambda=(-\Delta_p)^su$. By the compact embedding of $W^{s,p}_0(\Omega)$ in $L^1(\Omega)$, up to subsequence, we have that $u_n\to u$ in $L^1(\R^N)$ and thus $u_n(x)\to u(x)$ for a.e. $x\in \R^N$. In particular, letting
\[
w_n(x, y):=\frac{|u_n(x)-u_n(y)|^{p-2}(u_n(x)-u_n(y))}{|x-y|^{\frac{N+ps}{p'}}}\quad w(x, y):= \frac{|u(x)-u(y)|^{p-2}(u(x)-u(y))}{|x-y|^{\frac{N+ps}{p'}}},
\]
we have that
\[
w_n(x, y)\to w(x, y)\quad \text{for a.e. $(x, y)\in \R^{2N}$}.
\]
Moreover, since  $\|w_n\|_{L^{p'}(\R^{2N})}^{p'}=[u_n]_{s,p}^p\leq C$, we can assume that $w_n\weakto h$ in $L^{p'}(\R^{2N})$. Due to the pointwise convergence above, we thus have (see \cite[Lemme 4.8, Ch.1]{K}) $h=w$, and using the representation \eqref{splap} and the fact that
\[
\frac{\varphi(x)-\varphi(y)}{|x-y|^{\frac{N+ps}{p}}}\in L^p(\R^{2N}),\qquad \forall \varphi\in W^{s,p}_0(\Omega),
\]
we obtain through $(-\Delta_p)^s u_n\weakto \Lambda$ in $W^{-s, p'}(\Omega)$ and $w_n\weakto w$ in $L^{p'}(\R^{2N})$,
\[
\begin{split}
\langle \Lambda, \varphi\rangle &=\lim_n\langle (-\Delta_p)^s u_n, \varphi\rangle=\lim_n\int_{\R^{2N}}w_n(x, y)\frac{\varphi(x)-\varphi(y)}{|x-y|^{\frac{N+ps}{p}}}dx\, dy\\
&=\int_{\R^{2N}}w(x, y)\frac{\varphi(x)-\varphi(y)}{|x-y|^{\frac{N+ps}{p}}}dx\, dy=\langle (-\Delta_p)^s u, \varphi\rangle,
\end{split}
\]
for all $\varphi\in W^{s,p}_0(\Omega)$, proving the claim.
\end{proof}

Notice that this latter property is peculiar of the nonlocal setting, since for any $p\neq 2$ the corresponding local operator (the well known $p$-Laplacian) is {\em not} a weak-to-weak continuous operator.

\subsection{The energy functional}
The energy functional
$\J:W^{s,p}_0(\Omega)\to{\mathbb R}$
formally associated with problem \eqref{frac1} is
\begin{equation*}
\J(u):= \frac{1}{p}\int_{\R^{2N}}\frac{|u(x)-u(y)|^p}{|x-y|^{N+ps}}\, dx\, dy-\frac{\lambda}{r}\int_\Omega |u|^rdx-\frac{\mu}{q}\int_\Omega\frac{|u|^q}{|x|^\alpha}dx.
\end{equation*}
In order to justify that $\J$ is well defined and of class $C^1$ we need
some preliminary results.

\vskip2pt
\noindent

\begin{lemma}[Hardy-type functionals]
	\label{sob-func}
	Let $0\leq \alpha\leq ps<N$.
	Then, for every $p\leq q\leq p^*_\alpha,$
	$$
	H_q:W^{s,p}_0(\Omega)\to{\mathbb R},\qquad  H_q(u):=\int_\Omega\frac{|u|^{q}}{|x|^{\alpha}}dx,
	$$
	is of class $C^1$ with
	\begin{equation}
	\label{deriveta}
	\langle H_q'(u),\varphi\rangle=\int_\Omega\frac{|u|^{q-2}u\, \varphi}{|x|^{\alpha}}\, dx,\qquad \text{for every $u,\varphi \in W^{s,p}_0(\Omega)$.}
		\end{equation}
	 Moreover, for any $p\leq q< p^*_\alpha$,
	$H_q$ is sequentially weakly continuous and for any $p\leq q\leq p^*_\alpha$, the operator $H'_q:W^{s,p}_0(\Omega)\to W^{-s, p'}(\Omega)$ is sequentially weak-to-weak continuous.
\end{lemma}
\begin{proof}
For every $p\leq q\leq p^*_\alpha,$ it follows by Lemma~\ref{sobhardy} that
\begin{equation*}
\int_\Omega\frac{|u|^{q}}{|x|^{\alpha}}dx=\int_\Omega\frac{|u|^{q}}{|x|^{\alpha q/p^*_\alpha}}\frac{1}{|x|^{\alpha (1-q/p^*_\alpha)}}dx
\leq C\Big(\int_\Omega\frac{|u|^{p^*_\alpha}}{|x|^{\alpha}}dx\Big)^{q/p^*_\alpha}\leq C[u]_{s,p}^q.
\end{equation*}
Then, as it can be readily checked, \eqref{deriveta}
holds and $H_q'(u)\in W^{-s,p'}(\Omega)$ for all $u\in W^{s,p}_0(\Omega)$, since
$$
|\langle H_q'(u),\varphi\rangle|\leq C[u]_{s,p}^{q-1}[\varphi]_{s,p},\quad
\text{for all $\varphi \in W^{s,p}_0(\Omega)$.}
$$
Furthermore,  $H_q':W^{s,p}_0(\Omega)\to W^{-s,p'}(\Omega)$ is strong-to-strong continuous, proving that $H_q\in C^1$.
Let us prove the last assertions. If $u_n\weakto u$ in $W^{s,p}_0(\Omega)$, up to a subsequence, $u_n\to u$ in $L^\sigma(\Omega)$
for every $1\leq\sigma<p^*$ and a.e. in $\R^N$. Then, we get
\[
\begin{split}
\int_\Omega\frac{|u_n-u|^{q}}{|x|^{\alpha}}dx
& = \int_\Omega\frac{|u_n-u|^{\frac{\alpha}{s}}}{|x|^{\alpha}}|u_n-u|^{q-\frac{\alpha}{s}}dx \leq\Big(\int_\Omega\frac{|u_n-u|^{p}}{|x|^{ps}}dx\Big)^{\frac{\alpha}{ps}}
\Big(\int_\Omega |u_n-u|^{\sigma} dx\Big)^{\frac{ps-\alpha}{ps}}, \\
& \leq C\Big(\int_\Omega |u_n-u|^{\sigma} dx\Big)^{\frac{ps-\alpha}{ps}},
\end{split}
\]
where $\sigma:=(q-\frac{\alpha}{s})\frac{ps}{ps-\alpha}\in [1,p^*[$, which shows the weak continuity of $H_q$. Finally, to prove the weak-to-weak continuity of $H'_q$, we proceed as in Lemma \ref{wtwcont}. Observe that $\{u_n\}_n$ is bounded in $W^{s,p}_0(\Omega)$, and thus by the strong continuity of $H_q'$, so is $\{H_q'(u_n)\}_n\subseteq W^{-s,p'}(\Omega)$. Up to subsequences, we can assume that $H'_q(u_n)\weakto \Lambda\in W^{-s, p'}(\Omega)$ and by a standard sub-subsequence argument it suffices to show that $H'_q(u)=\Lambda$. Notice that  $\{|u_n|^{q-2}u_n\}_n$ is bounded in $L^{q'}(\Omega, dx/|x|^\alpha)$ due to H\"older and Hardy-Sobolev inequality, so that we can assume that $|u_n|^{q-2}u_n\weakto h$ in $L^{q'}(\Omega, dx/|x|^\alpha)$. Since $|u_n|^{q-2}u_n\to |u|^{q-2}u$ a.e., we get that $h=|u|^{q-2}u$. But then the representation \eqref{deriveta} implies, for any $\varphi\in W^{s,p}_0(\Omega)\subseteq L^q(\Omega, dx/|x|^{\alpha})$,
\[
\langle \Lambda, \varphi\rangle =\lim_n\langle H'_q(u_n), \varphi\rangle=\lim_n\int_\Omega |u_n|^{q-2}u_n\, \varphi\, \frac{dx}{|x|^{\alpha}}=\int_\Omega\frac{|u|^{q-2}u\, \varphi}{|x|^{\alpha}}\, dx=\langle H'_q(u), \varphi\rangle.
\]
\end{proof}

Combining the last assertion of the previous Lemma  with Lemma \ref{wtwcont} provides the following weak continuity result.
\begin{corollary}
\label{wtwcor}
For any  $\lambda, \mu\in \R$, $1\leq r\leq p^*$ and $1\leq q\leq p^*_\alpha$, $\J\in C^1$ and $\J':W^{s,p}_0(\Omega)\to W^{-s, p'}(\Omega)$ is both a strong-to-strong and weak-to-weak continuous operator.
\end{corollary}

This corollary justifies the definition of the Nehari manifold associated to $\J$ as
\[
{\mathcal N}=\{u\in W^{s, p}_0(\Omega)\setminus \{0\}: \langle \J'(u), u\rangle=0\}
\]
with subsets
\[
{\mathcal N}_+={\mathcal N}\cap \{u\geq 0\} \qquad {\mathcal N}_-={\mathcal N}\cap \{u\leq 0\}.
\]
Let $u^+=\max \{0, u\}$, $u^-=\min\{ u, 0\}$.
\begin{lemma}
\label{25}
For any $u\in W^{s,p}_0(\Omega)$ it holds
\begin{equation}
\label{b}
\langle (-\Delta_p)^su^\pm, u^\pm\rangle \leq \langle (-\Delta_p)^su, u^\pm\rangle\leq \langle (-\Delta_p)^su, u\rangle
\end{equation}
with strict inequality as long as $u$ is sign-changing.
\end{lemma}

\begin{proof}
We sketch the proof for $u^+$, the other one being identical. The statement follows by integration of the pointwise inequalities
\begin{equation}
\label{glk}
\phi_p(a_+-b_+)(a_+-b_+)\leq \phi_p(a-b)(a_+-b_+)\leq \phi_p(a-b)(a-b),
\end{equation}
where $\phi_p(t)=|t|^{p-2}t$. Notice that $\phi_p$ is strictly increasing and $t\mapsto t_+$ is non-decreasing, so that we can suppose $a>b$. Then $a_+-b_+\leq a-b$, with strict inequality as long as $b<0$, and thus the monotonicity of $\phi_p$ gives the conclusion. Finally, if $u$ is sign changing, $\{(x, y): u(x)>0>u(y)\}$ has positive measure in $\R^{2N}$ and on such set the previous inequality is strict.
\end{proof}

\begin{remark}\rm
\label{remneh}
As a corollary of the previous lemma, let us remark another fundamental (and more impactful) difference from the local case. Clearly, if $u$ is a sign-changing solution to \eqref{frac1}, then $u\in {\mathcal N}\setminus \left({\mathcal N}_+\cup{\mathcal N}_-\right)$. However
\[
u\in {\mathcal N}\setminus \left({\mathcal N}_+\cup{\mathcal N}_-\right)\quad \Rightarrow\quad \text{$u^+\notin {\mathcal N}$ or $u^-\notin {\mathcal N}$},
\]
since, otherwise, the inequality
\[
\langle \J'(u^+), u^+\rangle+\langle \J'(u^-), u^-\rangle<\langle \J'(u), u^+\rangle+\langle \J'(u), u^-\rangle=\langle \J'(u), u\rangle,
\]
yields a contradiction.
\end{remark}
The following simple observation will be used throughout the paper.

\begin{lemma}
\label{lemmalb}
Let  $\alpha\leq ps<N$, $q$ and $r$ satisfy $p\leq q\leq p^*_\alpha$, $p\leq r\leq p^*$, with $\max\{q, r\}>p$. For any $\lambda, \mu$ satisfying
\[
\begin{cases}
\lambda>0,\  \mu>0&\text{if $\min\{q, r\}>p$},\\
0<\lambda<\lambda_1,\  \mu>0&\text{if $r=p$},\\
\lambda>0, \ 0<\mu<\lambda_{1,\alpha}&\text{if $q=p$},\\
\end{cases}
\]
there exists $\delta_0=\delta_0(\Omega, N, p, s, \alpha)>0$ such that for any $u\in W^{s,p}_0(\Omega)$ it holds
\[
\langle \J'(u), u\rangle\leq  0\quad \Rightarrow \quad [u]_{s,p}\geq \delta_0.
\]
\end{lemma}

\begin{proof}
Applying H\"older and Hardy-Sobolev inequality on the last two terms of
\[
\langle \J'(u), u\rangle=[u]_{s,p}^p-\lambda \int_\Omega |u|^r\, dx-\mu \int_\Omega\frac{|u|^q}{|x|^\alpha}\, dx,
\]
we obtain
\[
\langle \J'(u), u\rangle\geq
\begin{cases}
\left(1-\dfrac{\lambda}{\lambda_1}-C[u]_{s, p}^{q-p}\right)[u]_{s, p}^p&\text{if $r=p$},\\[10pt]
\left(1-\dfrac{\mu}{\lambda_{1,\alpha}}-C[u]_{s, p}^{r-p}\right)[u]_{s, p}^p&\text{if $q=p$},\\[10pt]
\left(1-C [u]_{s, p}^{r-p}-C [u]_{s, p}^{q-p}\right)[u]_{s, p}^p&\text{if $\min\{q, r\}>p$}.
\end{cases}
\]
The assumption $\langle J'(u), u\rangle\leq  0$ forces the parenthesis above to be non-positive, which provides the claimed lower bound.
\end{proof}

\subsection{Properties of Hardy optimizers}

In \cite{brasco} the existence and properties of solutions for the minimization problem \eqref{criticalfrachardyg} when $\alpha=0$ was investigated.
For $0\leq\alpha< ps$, one can get the following results \cite[Theorem 1.1]{mm}.

\begin{proposition}[Existence and properties]
	\label{extral}
Let $0\leq \alpha< ps<N$. Then the following facts hold.

(i) Problem \eqref{criticalfrachardyg} admits constant sign solutions, and any solution is bounded;

(ii) For every nonnegative $U_\alpha\in  D^{s,p}(\R^N)$ solving \eqref{criticalfrachardyg}, there exist $x_0\in\R^N$ and a non-increasing $u:\mathbb{R}^+\to\mathbb{R}$ such that $U_\alpha(x)=u(|x-x_0|)$;

(iii) Every nonnegative minimizer $U_\alpha\in  D^{s,p}(\R^N)$ of \eqref{criticalfrachardyg} weakly solves
\begin{eqnarray*}
(-\Delta)_p^s U_\alpha =S_\alpha\frac{U_\alpha^{p^*_\alpha-1}}{|x|^\alpha}\ \ \
{\rm in}\ \R^N.
\end{eqnarray*}
i.e.,
\[
\langle (-\Delta_p)^sU_\alpha, \varphi\rangle=S_\alpha\int_{\R^N}\frac{U_\alpha^{p^*_\alpha-1}}{|x|^\alpha}\, \varphi\, dx, \quad \forall \varphi\in D^{s,p}(\R^N)
\]
and the last integrand is absolutely integrable.

\end{proposition}

\noindent
Next we {\em fix} $N,p,s,\alpha$ and a positive radially symmetric  decreasing minimizer
$U_\alpha=U_\alpha(r)$ for $S_\alpha$ as in \eqref{criticalfrachardyg}.
By multiplying $U_\alpha$ by a positive constant, we may assume\footnote{notice that we are using here that $p^*_\alpha\neq p$ since $\alpha<ps$.} that
\begin{align}\label{eqrneq}
(-\Delta)_p^s U_\alpha = \frac{ U_\alpha^{p^*_\alpha-1}}{|x|^\alpha}\quad
\text{weakly in  $\R^N$}.
\end{align}
Testing this equation by $U_\alpha$ and using \eqref{criticalfrachardyg} yield
\begin{align}\label{eqrn}
[U_\alpha]_{s,p}^p= \int_{\R^N}\frac{U_\alpha^{p^*_\alpha}}{|x|^\alpha}dx=S_\alpha^{\frac{N-\alpha}{ps-\alpha}} .
\end{align}
In \cite{brasco} the asymptotic behaviour for $U_\alpha$ was obtained when $\alpha=0,$ while in \cite{mm} the asymptotics for $U_\alpha$ for all $0<\alpha< ps$ was derived by similar arguments.

\begin{lemma}[Optimal decay]
	\label{leesti}
There exist $c_1>0$ and $c_2>0$ such that
\begin{equation*}
\frac{c_1}{r^{{\frac{N-ps}{p-1}}}}\leq U_\alpha(r)\leq \frac{c_2}{r^{{\frac{N-ps}{p-1}}}},
\,\,\quad\text{for all $r\geq 1$.}
\end{equation*}
Furthermore, there exists $\theta>1$
such that 
\begin{equation}
\label{ut}
U_\alpha(\theta r)\leq \frac{1}{2}U_\alpha(r)\quad \text{for all $r\geq 1$}.
\end{equation}
\end{lemma}

\noindent
For any $\eps>0$, the function
\begin{equation}
\label{Du}
U_{\alpha,\eps}(x):=\eps^{-\frac{N-ps}{p}}U_\alpha \left(\frac{x}{\eps}\right)
\end{equation}
is also a minimizer for $S_\alpha$ satisfying \eqref{eqrneq}.
We note that $c_1,c_2,\theta$ are universal since
we {\em fixed} $N,p,s,\alpha, U_\alpha$. In general they depend upon these entries.

\subsection{Truncations}
In what follows $0\leq \alpha<ps<N$, $U_\alpha$ is a fixed minimizer for \eqref{criticalfrachardyg}, $\theta$ is the constant in Lemma
\ref{leesti} depending only on $N$, $p$, $s$, $\alpha$ and $U_\alpha$. For every $\delta\geq \eps >0$, let us set
$$
m_{\eps,\delta}:=\frac{U_{\alpha,\eps}(\delta)}{U_{\alpha,\eps}(\delta)-U_{\alpha,\eps}(\theta\delta)}.
$$
Due to \eqref{ut} and the definition \eqref{Du}, it readily follows $m_{\eps, \delta}\leq 2$.
Furthermore, let us set
\begin{eqnarray*}
  g_{\eps,\delta}(t)=
  \left\{ \arraycolsep=1.5pt
\begin{array}{lll}
0,\ \  \ &
{\rm if}\ 0\leq t\leq U_{\alpha,\eps}(\theta\delta), \\[2mm]
m_{\eps,\delta}^p(t-U_{\alpha,\eps}(\theta\delta)),\ \  \ &
{\rm if}\  U_{\alpha,\eps}(\theta\delta)\leq t\leq U_{\alpha,\eps}(\delta), \\[2mm]
t+U_{\alpha,\eps}(\delta)(m_{\eps,\delta}^{p-1}-1), \ \ \quad & {\rm if}\ t\geq U_{\alpha,\eps}(\delta),
\end{array}
\right.
\end{eqnarray*}
as well as
\begin{eqnarray*}
G_{\eps,\delta}(t)=\int_0^t  g'_{\eps,\delta}(\tau)^{\frac{1}{p}}\, d\tau=
  \left\{ \arraycolsep=1.5pt
\begin{array}{lll}
0,\ \  \ &
{\rm if}\ 0\leq t\leq U_{\alpha,\eps}(\theta\delta), \\[2mm]
m_{\eps,\delta} (t-U_{\alpha,\eps}(\theta\delta)),\ \  \ &
{\rm if}\   U_{\alpha,\eps}(\theta\delta)\leq t\leq U_{\alpha,\eps}(\delta), \\[2mm]
t , \ \ \quad & {\rm if}\ t\geq U_{\alpha,\eps}(\delta).
\end{array}
\right.
\end{eqnarray*}
The functions $g_{\eps,\delta}$ and $G_{\eps,\delta}$ are nondecreasing and absolutely continuous. Consider now the radially symmetric nonincreasing function
\begin{eqnarray}\label{uepsdeltadef}
u_{\alpha,\eps,\delta}(r) :=  G_{\eps,\delta}(U_{\alpha,\eps}(r)),
\end{eqnarray}
which satisfies
\begin{eqnarray*}
u_{\alpha,\eps,\delta}(r)=  \left\{ \arraycolsep=1.5pt
\begin{array}{ll}
U_{\alpha,\eps}(r),\ \  \ &
{\rm if}\ r\leq   \delta,  \\[1mm]
0,\ \  \ &{\rm if}\   r\geq\theta\delta.
\end{array}
\right.
\end{eqnarray*}

\noindent
Then $u_{\alpha,\eps,\delta}\in W_{0}^{s,p}(\Omega)$, for
any $\delta<\theta^{-1}\delta_\Omega:=\theta^{-1}{\rm dist}(0,\partial\Omega)$.
We have the following estimates.

\begin{lemma}[Norm estimates I]
	\label{lebn}
There exists $C>0$ such that, for any $0<2\eps\leq\delta<\theta^{-1}\delta_\Omega$, there holds
\begin{align}
	\label{estqa}
	[u_{\alpha,\eps,\delta}]_{s,p}^p\leq S_\alpha^{\frac{N-\alpha}{ps-\alpha}}+C\left(\frac{\eps}{\delta}\right)^{ \frac{N-ps}{p-1} },
\end{align}
and
\begin{align}\label{estqaf}
 \int_{\R^N}\frac{u_{\alpha,\eps,\delta}(x)^{p^*_\alpha}}{|x|^\alpha}\, dx\geq S_\alpha^{\frac{N-\alpha}{ps-\alpha}}-C\left(\frac{\eps}{\delta}\right)^{\frac{N-\alpha}{p-1}}.
\end{align}
Moreover, for any $\beta>0$, there exists $C_\beta$ such that
\begin{align}
\label{estqal}
 \int_{\R^N}u_{\alpha,\eps,\delta}(x)^{\beta}\, dx\geq C_{\beta}
 \left\{ \arraycolsep=1.5pt
 \begin{array}{ll}
 \eps^{N-\frac{N-ps}{p}\beta}|\log\frac{\eps}{\delta}|,\ \  \ &
 {\rm if}\ \beta=\frac{p^*}{p'},\\[2mm]
 \eps^{\frac{N-ps}{p(p-1)}\beta}\delta^{N-\frac{N-ps}{p-1}\beta},\ \  \ &
 {\rm if}\ \beta<\frac{p^*}{p'},\\[2mm]
 \eps^{N-\frac{N-ps}{p}\beta},\ \  \ &
 {\rm if}\ \beta>\frac{p^*}{p'}.
 \end{array}
 \right.
\end{align}
\end{lemma}
\begin{proof}
The assertions follow as in the proof of  \cite[Lemma 2.7]{mpsy}. For the readers' convenience, we prove it here. Testing \eqref{eqrneq} with $g_{\eps,\delta}(U_{\alpha,\eps})\in W_{0}^{s,p}(\Omega)$ and using \cite[Lemma A.2]{bp} jointly with \eqref{eqrn} yields
\begin{align*}
[G_{\eps,\delta}(U_{\alpha,\eps})]_{s,p}^{p} \leq& \int_{\R^{2N}}
\frac{|U_{\alpha,\eps}(x)-U_{\alpha,\eps}(y)|^{p-2}\left(U_{\alpha,\eps}(x)-U_{\alpha,\eps}(y)\right)\left(g_{\eps,\delta}(U_{\alpha,\eps}(x))-g_{\eps,\delta}(U_{\alpha,\eps}(y))\right)}
{|x-y|^{N+ps}}\, dx\, dy\nonumber\\
 =&S_\alpha^{\frac{N-\alpha}{ps-\alpha}}+\int_{\mathbb{R}^{N}}\left[g_{\eps,\delta}(U_{\alpha,\eps})-U_{\alpha,\eps}\right]
\frac{U_{\alpha,\eps}^{p_{\alpha}^{\ast}-1}}{|x|^{\alpha}}\, dx.
\end{align*}
Moreover, by Lemma \ref{leesti}, it holds that
\begin{align}
g_{\eps,\delta}(t)-t &\leq U_{\alpha,\eps}(\delta)m_{\eps,\delta}^{p-1}
=\frac{1}{\eps^{\frac{N-ps}{p}}}U_\alpha\left(\frac{\delta}{\eps}\right)\Big[ 1-\frac{U_\alpha(\frac{\theta\delta}{\eps})}{U_\alpha\left(\frac{\delta}{\eps}\right)}\Big]^{-(p-1)}
 \leq 2^{p-1}c_2\frac{\eps^{\frac{N-ps}{p(p-1)}}}{\delta^{\frac{N-ps}{p-1}}}, \notag
\end{align}
and
\begin{align*}
\int_{\mathbb{R}^{N}} \frac{U_{\alpha,\eps}^{p_{\alpha}^{\ast}-1}}{|x|^{\alpha}}dx
 =&\eps^{\frac{N-ps}{p}}\int_{\mathbb{R}^{N}} \frac{U_\alpha^{p_{\alpha}^{\ast}-1}}{|z|^{\alpha}}dz=\O(\eps^{\frac{N-ps}{p}} ).
\end{align*}
This gives estimate \eqref{estqa}. On the other hand,
\begin{align}\label{wjc3}
\int_{\mathbb{R}^{N}}\frac{u_{\alpha,\eps,\delta}^{p_{\alpha}^{\ast}}}{|x|^{\alpha}}\, dx \geq \int_{B_{\delta}(0)}\frac{U_{\alpha,\eps}^{p_{\alpha}^{\ast}}}{|x|^{\alpha}}\, dx
 =S_\alpha^{\frac{N-\alpha}{ps-\alpha}}-\int_{\mathbb{R}^{N}\setminus B_{\frac{\delta}{\eps}}(0)}\frac{U_\alpha^{p_{\alpha}^{\ast}}}{|z|^{\alpha}}\, dz.
\end{align}
Since $\eps\leq\frac{\delta}{2}$, a simple calculation using Lemma \ref{leesti} yields that
\begin{align}\label{wjc4}
\int_{\mathbb{R}^{N}\setminus B_{\frac{\delta}{\eps}}(0)}\frac{U_\alpha^{p_{\alpha}^{\ast}}}{|z|^{\alpha}}\, dz\leq C\int_{\delta/\eps}^{+\infty} r^{N-1-\alpha}r^{p^*_\alpha\frac{N-ps}{p-1}}\, dr=C\int_{\delta/\eps}^{+\infty} r^{1-\frac{N-\alpha}{p-1}}\, dr=\O\Big(\Big(\frac{\eps}{\delta}\Big)^{\frac{N-\alpha}{p-1}}\Big).
\end{align}
Then \eqref{wjc3} and \eqref{wjc4} yield estimate \eqref{estqaf}.
Finally, we have
\[
\int_{\mathbb{R}^{N}}u_{\alpha,\eps,\delta}^{\beta}\, dx  \geq
\int_{B_{\delta}(0)}U_{\alpha,\eps}^{\beta}\, dx
 = \eps^{N-\frac{N-ps}{p}\beta}\int_{B_{\frac{\delta}{\eps}}(0)}U_\alpha^{\beta}\, dz\geq C\eps^{N-\frac{N-ps}{p}\beta}\int_{1}^{\frac{\delta}{\eps}}r^{N-\frac{N-ps}{p-1}\beta-1}\, dr
\]
and an explicit calculation provides \eqref{estqal}.
\end{proof}

We also have the following estimate.

\begin{lemma}[Norm estimates II]
	\label{lebn2}
  For any $\beta>0$
	and $0<2\eps\leq \delta< \theta^{-1}\delta_\Omega$, we have
	\begin{equation}
	\label{setqab}
	\int_\Omega u_{\alpha,\eps,\delta}^{\beta}dx\leq C_\beta
	\left\{ \arraycolsep=1.5pt
 \begin{array}{ll}
 \eps^{N-\frac{N-ps}{p}\beta}|\log\frac{\eps}{\delta}|,\ \  \ &
 {\rm if}\ \beta=\frac{p^*}{p'},\\[2mm]
 \eps^{\frac{N-ps}{p(p-1)}\beta}\delta^{N-\frac{N-ps}{p-1}\beta},\ \  \ &
 {\rm if}\ \beta<\frac{p^*}{p'},\\[2mm]
 \eps^{N-\frac{N-ps}{p}\beta},\ \  \ &
 {\rm if}\ \beta>\frac{p^*}{p'}.
 \end{array}
 \right.
	\end{equation}
\end{lemma}

\begin{proof}
	From the definition of $u_{\alpha,\eps,\delta}$,
	we have
\[
	u_{\alpha,\eps,\delta}(x)
	=
	\left\{ \arraycolsep=1.5pt
	\begin{array}{lll}
	0,\ \  \ &
	{\rm if}\ |x|\geq \theta \delta \\[2mm]
	m_{\eps,\delta} (U_{\alpha,\eps}(x)-U_{\alpha,\eps}(\theta\delta)),\ \  \ &
	{\rm if}\   \delta\leq |x|\leq \theta\delta \\[2mm]
	U_{\alpha,\eps}(x) , \ \ \quad & {\rm if}\ |x|\leq \delta
	\end{array}
	\right. \quad \leq  \left\{ \arraycolsep=1.5pt
	\begin{array}{lll}
	0,\ \  \ &
	{\rm if}\ |x|\geq \theta \delta  \\[2mm]
	m_{\eps,\delta} U_{\alpha,\eps}(x) \ \  \ &
	{\rm if}\   \delta\leq |x|\leq \theta\delta \\[2mm]
	U_{\alpha,\eps}(x) , \ \ \quad & {\rm if}\ |x|\leq \delta.
	\end{array}
	\right.
\]
	Recall that $m_{\eps,\delta}\leq  2$, therefore it holds $u_{\alpha,\eps,\delta}\leq 2U_{\alpha,\eps}\chi_{B_{\theta\delta}}$. Taking into account that $U_\alpha\in L^\infty(\R^N)$, we have
	\[
	\begin{split}
	\int_\Omega u_{\alpha,\eps,\delta}^\beta\, dx&\leq C \int_{B_{\theta\delta}}  U^\beta_{\alpha,\eps}\, dx=C\eps^{N-\frac{ N-ps  }{p}\beta}\int_{B_{\frac{\theta\delta}{\eps}}}U^\beta_\alpha\, dy \nonumber\\
	&\leq C\eps^{N-\frac{ N-ps  }{p}\beta}\int_{B_1}U_\alpha^\beta\, dy+C\eps^{N-\frac{ N-ps  }{p}\beta}\int_1^{ \frac{\theta\delta}{\eps}}r^{N-\frac{N-ps}{p-1}\beta-1} \, dt\nonumber\\
	&\leq C\|U_\alpha\|_\infty\eps^{N-\frac{ N-ps  }{p}\beta} +C
	\left\{ \arraycolsep=1.5pt
 \begin{array}{ll}
 \eps^{N-\frac{N-ps}{p}\beta}|\log\frac{\eps}{\delta}|,\ \  \ &
 {\rm if}\ \beta=\frac{p^*}{p'},\\[2mm]
 \eps^{\frac{N-ps}{p(p-1)}\beta}\delta^{N-\frac{N-ps}{p-1}\beta},\ \  \ &
 {\rm if}\ \beta<\frac{p^*}{p'},\\[2mm]
 \eps^{N-\frac{N-ps}{p}\beta},\ \  \ &
 {\rm if}\ \beta>\frac{p^*}{p'}.
 \end{array}
 \right.
\end{split}
\]
and in the range $\delta>2\eps$ the last term is always greater than the first, giving \eqref{setqab}.
\end{proof}

\subsection{Compactness results}\label{compactness}

\noindent
We first recall the following
\begin{definition}\label{defps}
Let $c\in\mathbb{R}$, $E$ be a Banach space and $J\in  C^1(E,\mathbb{R})$.

(i) $\{u_k\}_{k\in{\mathbb N}}\subset E$ is a $(PS)_c$-sequence for $J$ if $J(u_k)=c+o_k(1)$ and $J'(u_k)=o_k(1)$ in $E^\ast$.

(ii) $J$ satisfies the $(PS)_c$-condition if any $(PS)_c$-sequence
for $J$ has a convergent subsequence.
\end{definition}

\noindent
We will need a slight modification of the classical Br\'ezis-Lieb Lemma.

\begin{lemma}[Br\'ezis-Lieb]
	\label{blle}
Let $\{q_k\}_k$ be such that $q_k\geq 1$, $q_k\to q$ and let
$\{f_k\}_{k}\subset L^{q_k}(\mathbb{R}^m)$ be a bounded sequence such that $f_k\to f$  almost everywhere.
Then
$$
\|f_k\|_{L^{q_k}(\mathbb{R}^m)}^{q_k}-\|f_k-f\|_{L^{q_k}(\mathbb{R}^m)}^{q_k}
=\|f\|_{L^q(\mathbb{R}^m)}^q+o_k(1).
$$
\end{lemma}

\begin{proof}
Observe that the elementary inequality
\[
||a+b|^{q_k}-|a|^{q_k}|\leq \eps|a|^{q_k}+C_\eps|b|^{q_k}
\]
holds true  with a constant $C_\eps$ independent of $k$ for sufficiently large $k$\footnote{The inequality is trivial for $a=0$ and  dividing by $|a|\neq 0$ reduces to $h(t):=||1+t|^r-1|\leq \eps +C_\eps |t|^r$ for $r=q_k$. From $|1+t|^r\leq 2^r(1+|t|^r)$, we deduce $h(t)\leq 2^r+1+2^r|t|^r$ and the claim for $|t|\geq 1$ and $C_\eps\geq 2^{r+2}$. For $|t|\leq 1$, ${\rm Lip}(h)\leq r$, so that $h(t)\leq r |t|$ and by Young's inequality $r|t|\leq \eps+\frac{(r-1)^{r}}{r\eps^r} |t|^r$. In both cases $C_\eps$ is bounded if $r\geq 1$ is.}.  Then the classical proof  of  \cite[Theorem 1]{bl} carry over yielding the result.
\end{proof}

\noindent
\begin{lemma}[Convergences]
	\label{colema1}
Let $\{u_k\}_{k}\subset W^{s,p}_0(\Omega)$ be bounded and let $\{q_k\}_{k\in\mathbb{N}}$ be a sequence such that $p<q_k\leq p^*_\alpha$
and $q_k\to p^*_\alpha$ as $k\to\infty$. Then, up to a subsequence, we have
\begin{enumerate}
  \item $[u_k-u]_{s,p}^p=[u_k]_{s,p}^p-[u]_{s,p}^p+o_k(1)$.
  \item we have
  $$
  \int_\Omega\frac{|u_k-u|^{q_k}}{|x|^\alpha}\, dx= \int_\Omega\frac{|u_k|^{q_k}}{|x|^\alpha}\, dx- \int_\Omega\frac{|u|^{p^*_\alpha}}{|x|^\alpha}\, dx+o_k(1).
  $$
  \item for any $\varphi\in W^{s,p}_0(\Omega)$, we have
  $$
  \int_\Omega\frac{|u_k|^{q_k-2}u_k}{|x|^{\alpha}}\varphi \,dx= \int_\Omega\frac{|u|^{p^*_\alpha-2}u}{|x|^{\alpha}}\varphi \,dx+o_k(1).
  $$
\end{enumerate}
\end{lemma}
\begin{proof}
We can assume that $u_k\to u$ weakly in $W^{s,p}_0(\Omega)$ and pointwise a.e.. By choosing
$$
f_k=\frac{u_k(x)-u_k(y)}{|x-y|^{\frac{N+s p}{p}}},\quad\ f=\frac{u(x)-u(y)}{|x-y|^{\frac{N+s p}{p}}},\quad \ q_k\equiv p,\ \ \mbox{and}\ m=2N,
$$
(1) follows from Lemma \ref{blle}. Observe that by $q_k\leq p^*_\alpha$, H\"older and Hardy-Sobolev inequality,
\[
\int_{\Omega}\frac{|u_k|^{q_k}}{|x|^\alpha}\, dx=\int_{\Omega}\frac{|u_k|^{q_k}}{|x|^{\alpha\frac{q_k}{p^*_\alpha}}}\frac{1}{|x|^{\alpha\frac{p^*_\alpha-q_k}{p^*_\alpha}}}\, dx\leq \left(\int_{\Omega}\frac{|u_k|^{p^*_\alpha}}{|x|^\alpha}\, dx\right)^{\frac{q_k}{p^*_\alpha}}\left(\int_\Omega\frac{1}{|x|^{\alpha}}\, dx\right)^{1-\frac{q_k}{p^*_\alpha}}\leq C
\]
so that $\{u_k/|x|^{\alpha/p^*_\alpha}\}_k$ is a bounded sequence in $L^{p^*_\alpha}(\Omega)$ pointwise converging to $u/|x|^{\alpha/p^*_\alpha}$, and thus Lemma \ref{blle} again gives (2).
To prove (3), we let
\[
w_k:=\frac{|u_k|^{q_k-2}u_k}{|x|^{\frac{\alpha}{(p^*_\alpha)'}}},\qquad w:=\frac{|u|^{p^*_\alpha-2}u}{|x|^{\frac{\alpha}{(p^*_\alpha)'}}},
\]
and proceed as before obtaining that
\[
\int_\Omega|w_k|^{\frac{p^*_\alpha}{p^*_\alpha-1}}\, dx \leq
\int_\Omega \frac{|u_k|^{p^*_\alpha\frac{q_k-1}{p^*_\alpha-1}}}{|x|^{\alpha\frac{q_k-1}{p^*_\alpha-1}}}
\frac{1}{|x|^{\alpha(1-\frac{q_k-1}{p^*_\alpha-1})}}\, dx
 \leq \left(\int_\Omega \frac{|u_k|^{p^*_\alpha}}{|x|^{\alpha}}dx\right)^{\frac{q_k-1}{p^*_\alpha-1}}\leq C.
\]
Therefore $\{w_k\}_k$ is bounded  in $L^{(p^*_\alpha)'}(\Omega)$ and, up to subsequence, $w_k\weakto v$ in $L^{(p^*_\alpha)'}(\Omega)$. Since $w_k\to w$ pointwise a.e., $v=w$ and (3) follows noting that $\varphi/|x|^{\frac{\alpha}{p^*_\alpha}}\in L^{p^*_\alpha}(\Omega)$ for any $\varphi\in W^{s,p}_0(\Omega)$.
\end{proof}

 We now come to the compactness properties of the energy functional $\J$.

\begin{theorem}[Palais-Smale condition]
	\label{comth}
Let  $0\leq \alpha\leq ps<N$, $q$ and $r$ satisfy $ q\in [p, p^*_\alpha]$, $ r\in [p, p^*[$, with $\max\{q, r\}>p$. Then
\begin{enumerate}
  \item If $q< p^*_\alpha$ and $r<p^*$, then for any $\lambda, \mu>0$, $\J$ satisfies $(PS)_c$ for all $c\in {\mathbb R}$.
  \item If $q=p^*_\alpha>p$ and $r<p^*$, then for any $\lambda, \mu>0$, $\J$ satisfies $(PS)_c$ for all
\[
c<  \left(\frac{1}{p}-\frac{1}{p^*_\alpha}\right)\frac{S_\alpha^{\frac{N-\alpha}{ps-\alpha}}}{\mu^{\frac{p}{p^*_\alpha-p}}},
\]
   \item If $\alpha=ps$, $q=p^*_{ps}=p$ and $p<r<p^*$, then for any $\lambda>0$ and $0<\mu<S_{ps}$, $\J$ satisfies $(PS)_c$ for all $c\in {\mathbb R}$.
\end{enumerate}
\end{theorem}

\begin{proof}
Assume that $\{u_k\}_{k}$ is a $(PS)_c$-sequence of $\J$, that is
\begin{align*}
\J(u_k)=c+o_k(1),\quad \langle \J'(u_k),\varphi\rangle
=\langle w_k,\varphi\rangle,
\,\,\,\, \text{$w_k\to 0$ in $W^{-s,p'}(\Omega)$}.
\end{align*}
We have
\begin{align}\label{pscom0} [u_k]_{s,p}^p-\frac{\lambda p}{r}\int_\Omega|u_k|^r\, dx-\frac{\mu p}{q}\int_\Omega\frac{|u_k|^q}{|x|^\alpha}
\, dx&= p\J(u_k)= pc+o_k(1), \\
[u_k]_{s,p}^p- \lambda \int_\Omega|u_k|^r\, dx- \mu \int_\Omega\frac{|u_k|^q}{|x|^\alpha}\, dx&=\langle \J'(u_k),u_k\rangle=o_k(1)[u_k]_{s,p}, \notag
\end{align}
as $k\to\infty$. Then we get
\begin{equation}\label{pscom1}
 C+o_k(1)[u_k]_{s,p}\geq p\J(u_k)-\langle \J'(u_k),u_k\rangle =\mu\Big(1-\frac{p}{q}\Big)\int_{\Omega}\frac{|u_k|^q}{|x|^\alpha}\, dx+\lambda\Big(1-\frac{p}{r}\Big)\int_{\Omega}|u_k|^r\, dx.
\end{equation}
We first show that $\{u_k\}_k$ is bounded in $W^{s,p}_0(\Omega)$, splitting the proof in two cases.

\noindent
{\sc Case 1}: $\alpha<ps$.

\noindent

If both $q$ and $r$ are greater than $p$, then \eqref{pscom1} implies
\[
\int_{\Omega}\frac{|u_k|^q}{|x|^\alpha}\, dx\leq C(1+ [u_k]_{s, p}), \qquad \int_{\Omega}|u_k|^r\, dx\leq C(1+ [u_k]_{s, p})
\]
which, combined with  \eqref{pscom0}, gives
\begin{equation}
\label{relb}
[u_k]_{s,p}^p\leq C(1+[u_k]_{s,p})
\end{equation}
and the boundedness of $\{u_k\}_k$ in $W^{s,p}_0(\Omega)$ readily follows.

If $r=p$, then necessarily $q>p$ and H\"older's inequality with \eqref{pscom1} ensures
\[
\int_{\Omega}|u_k|^p\, dx = \int_{\Omega}\frac{|u_k|^p}{|x|^{\frac{p\alpha}{q}}}\, |x|^{\frac{p\alpha}{q}}\, dx \leq  C\Big(\int_{\Omega}\frac{|u_k|^q}{|x|^\alpha}\, dx\Big)^{\frac{p}{q}}\leq C(1+[u_k]_{s,p}^{\frac{p}{q}}),
\]
giving boundedness of $\{u_k\}_k$ by the same argument as before.

Finally, if $q=p$ and $r>p$, H\"older (recall that $\alpha<ps$) and Hardy inequalities give
\[
\begin{split}
\int_{\Omega}\frac{|u_k|^p}{|x|^\alpha}\, dx&\leq \left(\int_{\Omega}\frac{|u_k|^p}{|x|^{ps}}\, dx\right)^{\frac{\alpha}{ps}}\left(\int_{\Omega}|u_k|^p\, dx\right)^{1-\frac{\alpha}{ps}}\leq C\left(\int_{\Omega}\frac{|u_k|^p}{|x|^{ps}}\, dx\right)^{\frac{\alpha}{ps}}\left(\int_{\Omega}|u_k|^r\, dx\right)^{\frac{p}{r}-\frac{\alpha}{rs}}\\
&\leq C [u_k]^{\frac{\alpha}{s}}_{s, p}\left(1+ [u_k]_{s, p}\right)^{\frac{p}{r}-\frac{\alpha}{rs}}.
\end{split}
\]
Inserting this into \eqref{pscom0} and using \eqref{pscom1} provides
\[
 [u_k]^p_{s, p}\leq C\left[ 1+ [u_k]_{s, p}+  [u_k]^{\frac{\alpha}{s}}_{s, p}+ [u_k]^{\frac{p}{r}+\frac{\alpha}{s}\left(1-\frac 1 r\right)}_{s, p} \right],
\]
which again implies that $\{u_k\}_k$ is bounded in $W^{s,p}_0(\Omega)$ due to
\[
\alpha<ps\quad \Rightarrow\quad \frac{\alpha}{s}<\frac{p}{r}+\frac{\alpha}{s}\left(1-\frac 1 r\right)<p.
\]
{\sc Case 2}: $\alpha=ps$ (thus $q=p$) and $\mu<S_{ps}$.

\noindent
In this case we necessarily have $r>p$ and \eqref{pscom1} implies that
\[
\int_{\Omega}|u_k|^r\, dx\leq C(1+ [u_k]_{s, p}),
\]
so that Hardy's inequality and \eqref{pscom0} gives
\[
\left(1-\frac{\mu}{S_{ps}}\right) [u_k]^p_{s, p}\leq  [u_k]^p_{s, p}-\mu \int_{\Omega}\frac{|u_k|^p}{|x|^{ps}}\, dx\leq C(1+ [u_k]_{s, p}).
\]
Using $\mu<S_{ps}$ implies \eqref{relb} and the boundedness of $\{u_k\}_k$.

\vskip4pt
\noindent
Thus, $\{u_k\}_k$ is bounded, and passing if necessary to a subsequence such that $u_k\rightharpoonup v$ in $W_0^{s,p}(\Omega)$ as $k\to\infty$, all the conclusions of Lemma \ref{colema1} hold true. Corollary \ref{wtwcor} gives that   $\J'(v)=0$ and from $\langle \J'(v),v\rangle=0,$ we have that
\begin{equation}
\label{pscoml}
\J(v)
 =\lambda\left(\frac{1}{p}-\frac{1}{r}\right)\int_\Omega|v|^r\, dx+\mu \left(\frac{1}{p}-\frac{1}{q}\right)\int_\Omega\frac{|v|^{q}}{|x|^\alpha}\, dx\geq 0.
\end{equation}
Next we prove that $u_k\to v$ in $W^{s,p}_0(\Omega)$. By Lemma \ref{colema1} and the boundedness of $\{u_k\}_k$ we have
\[
o_k(1)=\langle \J'(u_k),u_k\rangle-\langle \J'(v),v\rangle
 =[u_k-v]_{s,p}^p-\mu\int_\Omega \frac{|u_k-v|^q}{|x|^\alpha}\, dx-\lambda\int_\Omega|u_k-v|^r\, dx,
 \]
which shows convergence if $q<p^*_\alpha$ and $r<p^*$. It suffices to analyze the case $q=p^*_\alpha$, $r<p^*$. Notice then that Lemma \ref{colema1} (for $q_k\equiv p^*_\alpha$) gives
\begin{align}\label{pscomu}
\J(u_k) &= \J(v)+\frac{1}{p}[u_k-v]_{s,p}^p-\frac{\mu}{p^*_\alpha}\int_\Omega \frac{|u_k-v|^{p^*_\alpha}}{|x|^\alpha}\, dx+o_k(1), \\
	\label{pscomau}
o_k(1)&=\langle \J'(u_k),u_k\rangle-\langle \J'(v),v\rangle
 =[u_k-v]_{s,p}^p-\mu\int_\Omega \frac{|u_k-v|^{p^*_\alpha}}{|x|^\alpha}\, dx+o_k(1).
\end{align}
We split the proof of the convergence as before.

\smallskip
\noindent
{\sc Case 1}: $\alpha<ps$, $r\in  [p, p^*[$ and $c<\left(\frac{1}{p}-\frac{1}{p^*_\alpha}\right)\frac{S_\alpha^{\frac{N-\alpha}{ps-\alpha}}}{\mu^{\frac{p}{p^*_\alpha-p}}}$.
\vskip1pt
\noindent
By virtue of \eqref{pscomau},
\[
 \frac{1}{p}[u_k-v]_{s,p}^p-\frac{\mu}{p^*_\alpha}\int_\Omega \frac{|u_k-v|^{p^*_\alpha}}{|x|^\alpha}\, dx=\left(\frac{1}{p}-\frac{1}{p^*_\alpha}\right)[u_k-v]_{s,p}^p+o_k(1).
 \]
From \eqref{pscoml} and \eqref{pscomu} we have
\[
 \frac{1}{p}[u_k-v]_{s,p}^p-\frac{\mu}{p^*_\alpha}\int_\Omega \frac{|u_k-v|^{p^*_\alpha}}{|x|^\alpha}\, dx\leq \J(u_k)+o_k(1)= c+o_k(1).
\]
Therefore
\[
\limsup_k\left(\frac{1}{p}-\frac{1}{p^*_\alpha}\right)[u_k-v]_{s,p}^p<\left(\frac{1}{p}-\frac{1}{p^*_\alpha}\right)\frac{S_\alpha^{\frac{N-\alpha}{ps-\alpha}}}{\mu^{\frac{p}{p^*_\alpha-p}}}.
\]
Using this and the fractional Sobolev-Hardy inequality, we have
\begin{align*}
o_k(1)
 \underset{\eqref{pscomau}}{=}&[u_k-v]_{s,p}^p-\mu\int_\Omega \frac{|u_k-v|^{p^*_\alpha}}{|x|^{\alpha}}\, dx \nonumber \geq [u_k-v]_{s,p}^p-\mu S_\alpha^{-\frac{p^*_\alpha}{p}}[u_k-v]_{s,p}^{p^*_\alpha}\nonumber\\
 =& [u_k-v]_{s,p}^p\left(1-\mu S_\alpha^{-\frac{p^*_\alpha}{p}}[u_k-v]_{s,p}^{p^*_\alpha-p}\right)\geq \omega_1[u_k-v]_{s,p}^p,
\end{align*}
for some $\omega_1>0$, giving the claim.

\smallskip
\noindent
{\sc Case 2}: $\alpha=ps$ and $\mu<S_{ps}$. We note that $p^*_\alpha=p$ and, from \eqref{pscomau}, we get
\begin{equation*}
o_k(1)
 =[u_k-v]_{s,p}^p-\mu\int_\Omega \frac{|u_k-v|^{p}}{|x|^{ps}}\, dx \nonumber\\
 \geq \Big(1-\frac{\mu}{S_{ps}}\Big)[u_k-v]_{s,p}^p,
\end{equation*}
which implies that $[u_k-v]_{s,p}=o_k(1)$.
\end{proof}

\begin{remark}\rm
\label{gps}
Inspecting the proof we see that the boundedness of Palais-Smale sequences follows solely from the conditions $|\J(u_k)|\leq C$ and  $|\langle \J'(u_k), u_k\rangle|\leq C [u_k]_{s, p}$.
\end{remark}

\section{Positive Solution}

\subsection{Existence of a Ground State}
The existence of a ground state positive solutions follows the standard Mountain pass approach. Let
\[
\Gamma=\left\{\gamma\in C^0([0, 1]; W^{s,p}_0(\Omega)): \gamma(0)=0, \J(\gamma(1))<0\right\},
\]
\[
c_1=\inf_{\gamma\in \Gamma}\sup_{t\in [0, 1]}\J(\gamma(t)).
\]
Our assumptions on the parameters are the following
\begin{equation}
\label{assumptions}
\begin{split}
0\leq \alpha\leq  ps<N,\quad &q\in [p, p^*_\alpha], \quad r\in [p, p^*[,\qquad \max\{q, r\}>p,\\
&\begin{cases}
\lambda>0,\  \mu>0&\text{if $\min\{q, r\}>p$},\\
0<\lambda<\lambda_1,\  \mu>0&\text{if $r=p$},\\
\lambda>0, \ 0<\mu<\lambda_{1,\alpha}&\text{if $q=p$}.
\end{cases}
\end{split}
\end{equation}
From $\max\{q, r\}>p$ it is readily checked that for any given $u\in W^{s,p}_0(\Omega)$, $\J(tu)\to -\infty$  for $t\to +\infty$.
Proceeding as in Lemma \ref{lemmalb} we obtain the lower bounds
\[
\J (u)\geq
\begin{cases}
\left(\dfrac{1}{p}-C[u]_{s, p}^{r-p}-C[u]_{s,p}^{q-p}\right)[u]_{s,p}^p&\text{if $\min\{r, q\}>p$},\\[10pt]
\left(\dfrac{1}{p}-\dfrac{\lambda}{p\lambda_1}-C[u]_{s,p}^{q-p}\right)[u]_{s,p}^p&\text{if $r=p<q$},\\[10pt]
\left(\dfrac{1}{p}-\dfrac{\mu}{p \lambda_{1,\alpha}}-C[u]_{s, p}^{r-p}\right)[u]_{s,p}^p&\text{if $q=p<r$},
\end{cases}
\]
which in turn imply, under assumption \eqref{assumptions}, that
\[
\inf_{[u]_{s,p}=\varrho}\J(u)>0
\]
for sufficiently small $\varrho>0$.
Since the support of any $\gamma \in \Gamma$ intersects $\{[u]_{s,p}=\varrho\}$, it holds $c_1\geq \inf_{[u]_{s,p}=\varrho}\J(u)>0$ and in order to apply the Mountain Pass theorem it only remains to show that $\J$ satisfies the $(PS)_{c_1}$ condition. By Theorem \ref{comth} this is certainly true in the subcritical or Hardy critical case when $\mu<S_{ps}$, so it suffices to consider the critical cases $q=p^*_\alpha>p$ (thus $\alpha<ps$).

\begin{lemma}[$c_1$ in the Palais-Smale range]
Let \eqref{assumptions} be fulfilled for $q=p^*_\alpha>p$. Then
\begin{equation}
\label{c1ps}
	 c_1<\left(\frac{1}{p}-\frac{1}{p^*_\alpha}\right)\frac{S_\alpha^{\frac{N-\alpha}{ps-\alpha}}}{\mu^{\frac{p}{p^*_\alpha-p}}}
\end{equation}
for any $r$ such that
\begin{equation}
\label{condr}
\begin{cases}
r\geq p&\text{if $N\geq p^2\, s$},\\
r>p^*-p'&\text{if $N<p^2\, s$}.
\end{cases}
\end{equation}
\end{lemma}

\begin{proof}
Notice that in any case considered in \eqref{condr}  it holds $r\geq p^*/p'$.
Without loss generality we can consider $\delta=1$ in the definition
	of $u_{\alpha,\eps,\delta}\in W^{s,p}_0(\Omega)$ given in formula \eqref{uepsdeltadef}.
	For any sufficiently small $1>\eps>0$, we apply \eqref{estqa}, \eqref{estqaf} and \eqref{estqal} with $\beta=r$, to obtain
	\[
	g_\eps(t):=\J (t u_{\alpha, \eps, 1})\leq \left(\frac{t^p}{p}-\mu\frac{t^q}{q}\right)S_\alpha^{\frac{N-\alpha}{ps-\alpha}} + Ct^p\eps^{\frac{N-ps}{p-1}}+Ct^q\eps^{\frac{N-\alpha}{p-1}}- \frac{1}{C}t^rh_r(\eps)
	\]
	for some large $C>0$, where
	\[
	h_r(\eps):=
	\begin{cases}
	\eps^{\frac{N}{p}}|\log\eps|&\text{if $r=p^*/p'$},\\
	\eps^{N-\frac{N-ps}{p}r}&\text{if $r>p^*/p'$}.
	\end{cases}
	\]
	It is readily seen from this estimate that there exists a small $\eps_0>0$ such that $g_\eps\to -\infty$ for $t\to +\infty$ uniformly for $\eps\in [0, \eps_0]$.  Moreover, $g_\eps(0)=0$, therefore there exists $A>a>0$ such that
	\[
	\sup_{t\geq 0} g_\eps =\sup_{t\in [a, A]} g_\eps(t),\qquad \forall \eps\in\  ]0, \eps_0].
	\]
	Observe that
	\[
	\sup_{t\geq 0} \left(\frac{t^p}{p}-\mu\frac{t^q}{q}\right)=\mu^{\frac{p}{p-p^*_\alpha}}\left(\frac{1}{p}-\frac{1}{p^*_\alpha}\right),
	\]
	so that, for some $C'$ also depending on $a, A$,
	\[
	\begin{split}
	\sup_{t\in [a, A]} g_\eps&\leq \sup_{t\geq 0} \left(\frac{t^p}{p}-\mu\frac{t^q}{q}\right)S_\alpha^{\frac{N-\alpha}{ps-\alpha}} +\sup_{t\in [a, A]}\Big[ Ct^p\eps^{\frac{N-ps}{p-1}}+Ct^q\eps^{\frac{N-\alpha}{p-1}}- \frac{1}{C}t^rh_r(\eps)\Big]\\
	&\leq \left(\frac{1}{p}-\frac{1}{p^*_\alpha}\right)\frac{S_\alpha^{\frac{N-\alpha}{ps-\alpha}}}{\mu^{\frac{p}{p^*_\alpha-p}}} +C'\eps^{\frac{N-ps}{p-1}}+C'\eps^{\frac{N-\alpha}{p-1}}- \frac{1}{C'}h_r(\eps).
	\end{split}
	\]
	Hence it suffices to prove that there exists a sufficiently small $\eps<\eps_0$ such that the last three term give a negative contribution. Using $\alpha<ps$, $\eps<1$ we have $\eps^{\frac{N-\alpha}{p-1}}<\eps^{\frac{N-ps}{p-1}}$
	so that a sufficient condition is that
	\[
	\lim_{\eps\to 0^+}\frac{h_r(\eps)}{\eps^{\frac{N-ps}{p-1}}}=+\infty, \qquad \text{for $r\geq p^*/p'$}.
	\]
	If $r>p^*/p'$, the previous condition is satisfied for
\[
	\frac{N-ps}{p-1}>N-\frac{N-ps}{p}r\quad \Leftrightarrow\quad r>p^*-p',
\]
while if $r=p^*/p'$, due to the logarithmic factor in $h_{r, \eps}$ it suffices that
\[
\frac{N-ps}{p-1}\geq \frac{N}{p}=N-\frac{N-ps}{p}r\quad \Leftrightarrow\quad r\geq p^*-p'.
\]
Now
\[
N\geq p^2\, s\quad \Rightarrow\quad p^*-p'\leq \frac{p^*}{p'}\leq p,
\]
and we obtain the full range $r\geq p$, while
\[
N<p^2\, s\quad \Rightarrow\quad   p^*-p'> \frac{p^*}{p'}>p
\]
giving the range $r>p^*-p'$ in this case.
\end{proof}
\noindent
Regarding \eqref{condr}, notice that when $N=p^2s$ then $p=p^*-p'$.\\
For the next corollary, recall that ${\mathcal N}_+=\{u\in {\mathcal N}: u\geq 0\}$.

\begin{corollary}
Let assumptions \eqref{assumptions} be fulfilled and, in the case $q=p^*_\alpha>p$, suppose that \eqref{condr} holds as well.
Then, there exists a nonnegative critical point $w$ solving $\J(w)=c_1$ and
\begin{equation}
\label{neh}
c_1=\inf_{u\in {\mathcal N}_+}\J (u)
\end{equation}
\end{corollary}

\begin{proof}
The existence of a critical point at level $c_1$ has already been discussed, so that it only remains to show that $w\in {\mathcal N}_+$ and \eqref{neh}.
Fix $u\neq 0$ and consider the function
\[
\R_+\ni s\mapsto \psi(s):=\J(s^{1/p} u)=\frac{s}{p}[u]_{s,p}^p-s^{\frac{r}{p}}\frac{\lambda}{r}\int_{\Omega}|u|^r\, dx-s^{\frac{q}{p}}\frac{\mu}{q}\int_{\Omega}\frac{|u|^q}{|x|^\alpha}\, dx.
\]
It holds
\[
\psi(0)=0, \quad \lim_{s\to +\infty}\psi(s)=-\infty,\quad \text{$\psi$ is concave}.
\]
Moreover, by the assumption $\max\{r, q\}>p$, $\psi$ is strictly concave. Therefore $\psi$ has a unique maximum $s_u$, which is positive due to
\[
\psi'(0)=
\begin{cases}
\frac{1}{p}[u]_{s,p}^p&\text{if $\min\{r, q\}>p$},\\
\frac{1}{p}\Big([u]_{s,p}^p-\lambda\|u\|_{L^p}^p\Big)&\text{if $r=p$},\\
\frac{1}{p}\Big([u]_{s,p}^p-\mu \|u/|x|^\alpha\|_{L^p}^p\Big)&\text{if $q=p$},
\end{cases}
\]
all of which are strictly positive due to \eqref{assumptions}. Therefore $\psi'(s)>0$ for $s<s_u$, $\psi'(s)<0$ for $s>s_u$. Changing variable $s=t^p$, this translates to
\begin{equation}
\label{prneh}
\begin{cases}
\text{$\J(t u)\to -\infty$ for $t\to +\infty$},\\
 \text{$t\mapsto \J(t u)$ has a unique positive maximum $t_u$},\\
 \text{$\langle \J'(tu), u\rangle (t-t_u)<0$ $\forall t\neq t_u$}.
 \end{cases}
\end{equation}
Hence, given $u\in {\mathcal N}$, $\J(u)=\sup_{t\geq 0} \J(t u)$, which readily implies
\[
c_1\leq \inf_{u\in {\mathcal N}}\J (u).
\]
Here we used the fact that
$$
c_1=\inf_{\gamma\in \Gamma}\sup_{t\in [0, 1]}\J(\gamma(t))\leq \inf_{u\in W_0^{s,p}(\Omega)}\sup_{t\geq0}\J(tu).
$$
The opposite inequality follows from the solvability of the minimax problem defining $c_1$, as the critical point $w$ at level $c_1$ certainly lies in ${\mathcal N}$. It remains to show that
\[
\inf_{u\in {\mathcal N}}\J (u)=\inf_{u\in {\mathcal N}_+}\J (u).
\]
i.e., that the ground state $w$ solving $J(w)=c_1$ can be chosen nonegative. Clearly the inequality $\geq$ above suffices. Since $\J$ is even and $ w\neq 0$, we may suppose without loss of generality that $w^+\neq 0$. By \eqref{b} we have
\[
\langle \J'(w^+), w^+ \rangle\leq \langle \J'(w), w^+\rangle=0,
\]
so that $t_{w^+}$ defined in \eqref{prneh} satisfies $t_{w^+}\leq 1$. But then
\[
\begin{split}
\inf_{u\in {\mathcal N}_+}\J (u)&\leq \J(t_{w^+} w^+)=t_{w^+}^r\left(\frac{\lambda}{p}-\frac{\lambda}{r}\right)\int_{\Omega}|w^+|^r\, dx+t_{w^+}^q\left(\frac{\mu}{p}-\frac{\mu}{q}\right)\int_{\Omega}\frac{|w^+|^q}{|x|^\alpha}\, dx\\
&\leq \left(\frac{\lambda}{p}-\frac{\lambda}{r}\right)\int_{\Omega}|w|^r\, dx+\left(\frac{\mu}{p}-\frac{\mu}{q}\right)\int_{\Omega}\frac{|w|^q}{|x|^\alpha}\, dx= J(w)=c_1=\inf_{u\in {\mathcal N}}\J (u).
\end{split}
\]
Finally, observe that $w^-\neq 0$ implies that the inequality in the second line of the previous chain is strict. Therefore the mountain pass solution must be of constant sign.
\end{proof}

\subsection{Qualitative and asymptotic properties}
\label{assection1}
In this section we will study the properties of the ground state in the quasi-critical approximation. In order to do so, we will assume that $\alpha<ps$, so that $p^*_\alpha>p$. In this setting, assumptions \eqref{assumptions} simplify to 
\begin{equation}
\label{assumptionsb}
\begin{split}
&0\leq \alpha<  ps<N,\quad q\in \ ]p, p^*_\alpha[\,, \quad r\in [p, p^*[\,,\\
&\begin{cases}
\lambda>0,\  \mu>0&\text{if $r>p$},\\
0<\lambda<\lambda_1,\  \mu>0&\text{if $r=p$}.
\end{cases}
\end{split}
\end{equation}
For $\rho\in \ ]0, p^*_\alpha-p[$\, , define $q_\rho=p^*_\alpha-\rho$ and
\begin{equation}
\label{Jrho}
J_\rho(u):= \frac{1}{p}[u]_{p, s}^p-\frac{\lambda}{r}\int_\Omega|u|^r\, dx-\frac{\mu}{q_\rho}\int_\Omega\frac{|u|^{q_\rho}}{|x|^\alpha}\, dx,
\end{equation}
with $w_\rho\geq 0$ being the corresponding ground state solution, i.e.
\[
{\mathcal N}^\rho=\{u\in W^{s,p}_0(\Omega)\setminus\{0\}:\langle J'_\rho (u), u\rangle=0\},\qquad c_{1, \rho}=\inf_{{\mathcal N}^\rho}J_\rho=J_\rho(w_\rho),
\]
so that, with the notations of the previous section, $c_{1,0}=c_1$.
Our aim is the behavior of the family $\{w_\rho\}_\rho$ as $\rho\to 0$. We begin with a general boundedness result, which will be useful also in the sign-changing case.

\begin{theorem}
Let $\Omega$ be bounded and $u\in W^{s,p}_0(\Omega)$ weakly solve $(-\Delta_p)^s u=f(x, u)$ in $\Omega$ for $f$ satisfying
\begin{equation}
\label{gcond}
|f(x, t)|\leq C\left(1+|t|^{p^*-1}+\frac{|t|^{p^*_\alpha-1}}{|x|^\alpha}\right),\qquad 0\leq \alpha<ps<N.
\end{equation}
Then $u$ is bounded and continuous on $\bar \Omega$.
\end{theorem}

\begin{proof}
By \cite[Theorem 3.1 and Theorem 3.13]{bp} it suffices to prove that $f(x, u)\in L^{\bar q}(\Omega)$ for some $\bar q>\frac{N}{ps}$. Given $k>0$, $t\geq 0$, and $\beta\geq 0$, we define $g_\beta(t):=t (t_k)^{\beta}$, where $t_k:=\min\{t,k\}$
and extend $g_\beta$ and $t_k$ as an odd function. It holds
\begin{equation}
\label{pao}
G_\beta(t):=\int_0^t g_\beta'(\tau)^{\frac{1}{p}}\, d\tau\geq \frac{p\,(\beta+1)^{1/p}}{p+\beta}\, g_{\frac{\beta}{p}}(t),
\end{equation}
and due to \cite[Lemma A.2]{bp} and \eqref{gcond}
\[
[G_\beta(u)]_{s,p}^p\leq \langle (-\Delta_p)^s u, g_\beta(u)\rangle\leq C\left(\int_{\Omega}|u||u_k|^\beta\, dx+\int_\Omega |u|^{p^*} |u_k|^\beta\, dx+\int_\Omega\frac{|u|^{p^*_\alpha}|u_k|^\beta}{|x|^\alpha}\, dx\right).
\]
 By Hardy-Sobolev's inequality with weight $1$ and $|x|^{-\alpha}$ and \eqref{pao}
we obtain, for $\gamma=0, \alpha$,
\begin{equation}
\label{pao2}
\left(\int_{\Omega} \frac{|u|^{p^*_\gamma}|u_k|^{\beta\frac{p^*_\gamma}{p}}}{|x|^\gamma}\, dx\right)^{\frac{p}{p^*_\gamma}}\leq C_\beta\left(\int_{\Omega}|u||u_k|^\beta\, dx+\int_\Omega |u|^{p^*} |u_k|^\beta\, dx+\int_\Omega\frac{|u|^{p^*_\alpha}|u_k|^\beta}{|x|^\alpha}\, dx\right).
\end{equation}
For any $K>1$, $\gamma=0$, $\alpha$, we deduce
\[
\begin{split}
\int_{\Omega}&\frac{|u|^{p^*_\gamma} \, |u_k|^{\beta}}{|x|^\gamma}\, dx\leq \int_{\{|u|<K\}}\frac{|u|^{p^*_\gamma}\, |u_k|^{\beta}}{|x|^\gamma}\, dx+\int_{\{|u|\geq K\}}\frac{|u|^{p^*_\gamma}\,  |u_k|^{\beta}}{|x|^\gamma}\, dx\\
&\leq K^{\beta}\int_{\Omega}\frac{|u|^{p^*_\gamma}}{|x|^\gamma}\, dx+\left(\int_{\{|u|\geq K\}}\frac{|u|^{p^*_\gamma}}{|x|^\gamma}\, dx\right)^{1-\frac{p}{p^*_\gamma}}\left(\int_{\Omega}\frac{|u|^{p^*_\gamma}\, |u_k|^{\beta\frac{p^*_\gamma}{p}}}{|x|^\gamma}\, dx\right)^{\frac{p}{p^*_\gamma}},
\end{split}
\]
while being $K\geq 1$
\[
\begin{split}
\int_{\Omega}&|u||u_k|^\beta\, dx\leq K^\beta\int_{\Omega}|u|+\int_{\{|u|\geq K\}}|u|^{p^*}\,  |u_k|^{\beta}\, dx\\
&\leq K^\beta\int_{\Omega}|u|+\left(\int_{\{|u|\geq K\}}|u|^{p^*}\, dx\right)^{1-\frac{p}{p^*}}\left(\int_{\Omega}|u|^{p^*}\, |u_k|^{\beta\frac{p^*}{p}}\, dx\right)^{\frac{p}{p^*}}.
\end{split}
\]
Since $u^{p^*_\gamma}\in L^1(\Omega, dx/|x|^\gamma)$ by Hardy-Sobolev, we can choose $K$ large enough so that
\begin{equation}
\label{condK}
C_\gamma\left(\int_{\{|u|\geq K\}}\frac{|u|^{p^*_\gamma}}{|x|^\gamma}\, dx\right)^{1-\frac{p}{p^*_\gamma}}<\frac 1 8, \qquad  \mbox{for}\ \gamma=0, \alpha,
\end{equation}
and reabsorb the corresponding terms on the right hand side of \eqref{pao2} (summed for $\gamma=0, \alpha$), to obtain
\[
\begin{split}
\left(\int_{\Omega} |u|^{p^*}|u_k|^{\beta\frac{p^*}{p}}\, dx\right)^{\frac{p}{p^*}}+\left(\int_{\Omega} \frac{|u|^{p^*_\alpha}|u_k|^{\beta\frac{p^*_\alpha}{p}}}{|x|^\alpha}\, dx\right)^{\frac{p}{p^*_\gamma}}&\leq C_\beta K^\beta \left(\int_\Omega|u|\, dx+\int_{\Omega}|u|^{p^*}\, dx+\int_{\Omega}\frac{|u|^{p^*_\alpha}}{|x|^\alpha}\, dx\right)\\
&\leq C_\beta K^\beta (1+[u]_{s,p}^p).
\end{split}
\]
Letting $k\to +\infty$ we deduce that $u\in L^q(\Omega, dx/|x|^\alpha)$ for every $q>1$, with norm depending only on the data  appearing in \eqref{gcond}, $K$ given in \eqref{condK}, $[u]_{s,p}^p$ and $q$. This directly implies that $|u|^{p^*-1}\in L^q(\Omega)$ for any $q>\frac{N}{ps}$. In order to prove that $f(x, u)\in L^{\bar q}(\Omega)$ for some $\bar q>\frac{N}{ps}$ it remains to show, by \eqref{gcond},  that
\begin{equation}
\label{pha}
v:=\frac{|u|^{p^*_\alpha-1}}{|x|^\alpha}\in L^{\bar q}(\Omega),\quad \text{for some $\bar q>\frac{N}{ps}$.}
\end{equation}
To this end, we choose  $\bar q\in\ ]\frac{N}{ps},\frac{N}{\alpha}[$ and $q>1$ so large that
\[
\alpha(\bar q-\frac{1}{q})q'<N.
\]
Then, H\"older inequality with exponents $q$ and $q'$ yields
\[
\int_\Omega v^{\bar q}\, dx=\int_{\Omega}\frac{|u|^{\bar q(p^*_\alpha-1)}}{|x|^{\frac{\alpha}{q}}}\frac{1}{|x|^{\alpha(\bar q-\frac{1}{q})}}\, dx
\leq \left(\int_{\Omega}\frac{u^{\bar q(p^*_\alpha-1)q}}{|x|^\alpha}\, dx\right)^{\frac{1}{q}}
\left(\int_{\Omega}\frac{1}{|x|^{\alpha(\bar q-\frac{1}{q})q'}}\, dx\right)^{1-\frac{1}{q}}.
\]
Both integrals are finite, thus \eqref{pha} and the boundedness and continuity in $\Omega$ of $u$  are proved. Finally, the regularity up to the boundary is proven in \cite{IMS} (notice that we are assuming $0\notin \partial\Omega$).
\end{proof}

\begin{remark}\rm
\label{remlbound}
For future purposes, it is worth outlining the dependence of $\|u\|_\infty$. We have
$$
\|u\|_\infty\leq C\|f(x, u)\|_{\bar q}^{1/(p-1)},
$$
for some universally chosen $\bar q>N/ps$, and the latter only depends on the data appearing in \eqref{gcond}, $[u]_{s,p}$ and $K$ chosen in \eqref{condK}.
\end{remark}

\begin{lemma}
Suppose \eqref{assumptionsb} and \eqref{condr} hold true. Then it holds
\begin{equation}
\label{convc}
\limsup_{\rho\to 0}c_{1, \rho}\leq  c_{1, 0}.
\end{equation}
\end{lemma}

\begin{proof}
Fix a sequence $\rho_n\to 0$, let $w_0$ solve the critical problem for $c_{1, 0}$ and let $t_n>0$ be such that $t_n w_0\in {\mathcal N}^{\rho_n}$, constructed through \eqref{prneh}. The family of functions $g_n(t):= J_{\rho_n}(t w_0)$ converges locally uniformly on $\R_+$ to $g_\infty(t):=J_0(t w_0)$ and
$g_n(t)\to -\infty$ for $t\to +\infty$ uniformly in $n$. Therefore, by the uniqueness of the maximum of $g_{n}$ and $g_\infty$, it follows that  $t_n={\rm Argmax}(g_n)\to {\rm Argmax}(g_\infty)=1$ as $n\to +\infty$. But then
\[
\limsup_{n} c_{1, \rho_n}\leq \limsup_n J_{\rho_n}(t_n w_0)=J_{0}(w_0)=c_{1, 0},
\]
since by dominated convergence
\[
\frac{t_n^{q_{\rho_n}}}{q_{\rho_n}}\int_{\Omega}\frac{|w_0|^{q_{\rho_n}}}{|x|^\alpha}\, dx \to \frac{1}{q_0}\int_{\Omega}\frac{|w_0|^{q_0}}{|x|^\alpha}\, dx.
\]
\end{proof}

\begin{remark}\rm
Even if not needed, proceeding as in \cite{gy} it can be proved that actually $c_{1, \rho}\to c_{1, 0}$.
\end{remark}

\begin{lemma}
\label{bsol}
Let $0\leq \alpha<ps<N$, $r\in [p, p^*[$ and let  $\rho_n>0$, $\rho_n\to 0$.  If  $u_n$ is a solution to \eqref{frac1} with $q=p^*_\alpha-\rho_n$ such that $\{J_{\rho_n}(u_n)\}_n$ is bounded,  then $\{u_n\}_n$ is bounded and up to subsequences weakly converges to a solution $u\in W^{s,p}_0(\Omega)$ of $J'_0(u)=0$.
\end{lemma}

\begin{proof}
As in Theorem \ref{comth}, the boundedness of  $\{J_{\rho_n}(u_n)\}_n$ and $J_{\rho_n}'(u_n)\equiv 0$ imply the boundedness of $\{u_n\}_n$ in $W^{s,p}_0(\Omega)$ (see also Remark \ref{gps}), so that, up to subsequences, we can assume $u_n\weakto u$ in $ W^{s,p}_0(\Omega)$ and pointwise a.e. in $\Omega$. Therefore, by Corollary \ref{wtwcor} and dominated convergence,  for every $\varphi\in W^{s,p}_0(\Omega)$,
\[
\langle J_0'(u), \varphi\rangle=\lim_n\langle J_0'(u_n), \varphi\rangle=\lim_n\langle J_0'(u_n)-J'_{\rho_n}(u_n), \varphi\rangle=\lim_n\int_\Omega\frac{|u_n|^{p^*_\alpha-1}-|u_n|^{p^*_\alpha-1-\rho_n}}{|x|^\alpha}\varphi\, dx=0,
\]
hence $J_0'(u)=0$, as claimed.
\end{proof}

\begin{theorem}
\label{equibound}
Let \eqref{assumptionsb} and \eqref{condr} hold. Then, the family $\{w_\rho\}_\rho$ is strongly compact in $W^{s, p}_0(\Omega)$ and any limit point of $\{w_\rho\}_\rho$ is a nontrivial critical point of minimal energy for $J_0$. Moreover, any sequence $\{w_{\rho_n}\}_n$ with $\rho_n\to 0$ admits an equi-bounded subsequence in $L^\infty$.
\end{theorem}

\begin{proof}
Fix a sequence $\rho_n\to 0$. The boundedness and weak convergence to a solution $w_0$ of $J_0'(w_0)=0$ of $\{w_{\rho_n}\}_n$ follows from \eqref{convc}, $J_{\rho_n}(w_{\rho_n})\geq 0$ and the previous Lemma. We claim that $\{w_{\rho_n}\}_n$ is strongly compact in $W^{s,p}_0(\Omega)$.
To this end, it suffices to modify the proof of Theorem \ref{comth}. Indeed, by the strong convergence of $\{w_{\rho_n}\}_n$ to $w_0$ in $L^r(\Omega)$ and the variable exponent Brezis-Lieb Lemma \ref{colema1}, point 2),  it holds
\[
\begin{split}
\J_{\rho_n}(w_{\rho_n}) &= \J_{0}(w_0)+\frac{1}{p}[w_{\rho_n}-w_0]_{s,p}^p-\frac{\mu}{p^*_\alpha-\rho_n}\int_\Omega \frac{|w_{\rho_n}-w_0|^{p^*_\alpha-\rho_n}}{|x|^\alpha}\, dx\\
&\qquad -\left(\frac{\mu}{p^*_\alpha-\rho_n}-\frac{\mu}{p^*_\alpha}\right)\int_\Omega\frac{|w_0|^{p^*_\alpha}}{|x|^\alpha}\, dx + o_n(1), \\
&= \J_{0}(w_0)+\frac{1}{p}[w_{\rho_n}-w_0]_{s,p}^p-\frac{\mu}{p^*_\alpha-\rho_n}\int_\Omega \frac{|w_{\rho_n}-w_0|^{p^*_\alpha-\rho_n}}{|x|^\alpha}\, dx + o_n(1)
\end{split}
\]
and similarly
\[
0=\langle \J_{\rho_n}'(w_{\rho_n}),w_{\rho_n}\rangle-\langle \J_0'(w_0),w_0\rangle
 =[w_{\rho_n}-w_0]_{s,p}^p-\mu\int_\Omega \frac{|w_{\rho_n}-w_0|^{p^*_\alpha-\rho_n}}{|x|^\alpha}\, dx+o_n(1).
\]
 By \eqref{convc} and \eqref{c1ps}, there exists $\eps>0$ such that for any sufficiently large $n$ it holds 
\[
J_{\rho_n}(w_{\rho_n})<c_{1,0}+\eps<\left(\frac{1}{p}-\frac{1}{p^*_\alpha-\rho_n}\right)\frac{S_\alpha^{\frac{N-\alpha}{ps-\alpha}}}{\mu^{\frac{p}{p^*_\alpha-p}}}.
\]
Taking advantage of the last three relations exactly as in the proof of Theorem \ref{comth}, we get
\[
\limsup_{n}\, [w_{\rho_n}-w_0]_{s,p}^p<\frac{S_\alpha^{\frac{N-\alpha}{ps-\alpha}}}{\mu^{\frac{p}{p^*_\alpha-p}}}
\]
and, similarly, we obtain the conclusion $w_{\rho_n}\to w_0$ in $W^{s,p}_0(\Omega)$.
Since $[w_{\rho_n}]_{s,p}$  is uniformly bounded from below by Lemma \ref{lemmalb}, it follows that $w_0\neq 0$. Moreover, the strong continuity of $(\rho, u)\mapsto J_\rho(u)$, together with \eqref{convc} and \eqref{neh}, provides the minimality property of $w_0$.
Finally, we prove the uniform boundedness of $\{w_{\rho_n}\}_n$. Observe that $w_{\rho_n}\to w_0$ both in $L^{p^*}(\Omega)$ and $L^{p^*_\alpha}(\Omega, dx/|x|^\alpha)$. Hence both the families
$\{w_{\rho_n}^{p^* }\}_n$ and $\{w_{\rho_n}^{p^*_\alpha}/|x|^\alpha\}_n$
are equi-integrable. Since each $w_{\rho_n}$ solves the equation  $(-\Delta_p)^s w_{\rho_n}=f_{\rho_n}(x, w_{\rho_n})$, with $\{f_{\rho_n}\}_n$ satisfying \eqref{gcond} uniformly in $n$, the equi-boundedness follows from Remark \ref{remlbound}.
\end{proof}

\section{Sign-changing solution}

We will now construct sign-changing solutions to the critical problem, following the procedure described in the introduction. 
\subsection{Subcritical case}\label{subcr}
Recall that  $u^+=\max \{0, u\}$, $u^-=\min\{ u, 0\}$ and let
\begin{equation}
\label{Nsc}
{\mathcal N}_{\rm sc}=\{ u\in {\mathcal N}: u^\pm\neq 0, \langle \J'(u), u^\pm\rangle =0\}.
\end{equation}
Clearly, any sign-changing solution to \eqref{frac1} belongs to ${\mathcal N}_{\rm sc}$. However, contrary to the local case, it is not even clear that ${\mathcal N}_{\rm sc}$ is nonempty. 
That this is actually so is the content of the following lemma.

\begin{lemma}\label{le51}
Let $u\in W^{s,p}_0(\Omega)$ be such that $u^\pm\ne 0$ and \eqref{assumptions} be fulfilled. Then the maximum
\[
\sup_{(t_1, t_2)\in \R^2} \J(t_1 u^+ +t_2 u^-)
\]
is attained at a unique $(\bar t_1, \bar t_2)\in \R^2_+$. Moreover, $(\bar t_1, \bar t_2)$ is a global maximum point if and only if
\begin{equation}
\label{maxcond}
\bar t_1\cdot \bar t_2>0, \qquad \bar t_1 u^++\bar t_2u^-\in {\mathcal N}_{\rm sc}\end{equation}
and
\begin{equation}
\label{tmon}
\langle \J'( u), u^+\rangle (\bar t_1-1)\geq 0, \qquad \langle \J'( u), u^-\rangle (\bar t_2-1)\geq 0.
\end{equation}
\end{lemma}

\begin{proof}
Letting $\Omega_{\pm}={\rm supp}(u^\pm)$ and write explicitly
\[
\begin{split}
\J(t_1 u^++t_2 u^-)=
|t_1|^pA_+&+
|t_2|^pA_- -|t_1|^rB_+-|t_2|^rB_--|t_1|^qC_+-|t_2|^qC_-\\
&+ \frac{2}{p}\int_{\Omega_+\times \Omega_-}\frac{|t_1 u^+(x)-t_2 u^-(y)|^p}{|x-y|^{N+ps}}\, dx\, dy
\end{split}
\]
where
\[
A_\pm=\frac{1}{p}\int_{\Omega_\mp^c\times \Omega_\mp^c}\frac{|u^+(x)-u^+(y)|^p}{|x-y|^{N+ps}}\, dx\, dy\quad B_\pm=\frac{\lambda}{r}\int_{\Omega_\pm}|u|^r\, dx\quad
C_\pm=\frac{\mu}{q}\int_{\Omega_\pm}\frac{|u|^q}{|x|^\alpha}\, dx.
\]
Thanks to $\max\{q, r\}>p$ and $u^\pm\ne 0$, it is readily checked that
\[
\lim_{|t_1|+|t_2|\to +\infty}\J(t_1 u^++t_2 u^-)=-\infty,
\]
so that a maximum exists.
Since $u^-\le 0$, analyzing the mixed integral term shows that the maximum must be attained on $t_1, t_2\geq 0$ (or on $t_1, t_2\leq 0$), which we will suppose henceforth. Let us consider the function
\[
\R^2_+\ni (s_1, s_2)\mapsto \psi(s_1, s_2)=\J(s_1^{1/p}u^++s_2^{1/p}u^-).
\]
It holds
\begin{equation}
\label{pas}
\begin{split}
\psi(s_1, s_2)&=s_1A_++s_2A_--s_1^{r/p}B_+-s_2^{r/p}B_--s_1^{q/p}C_+-s_2^{q/p}C_-\\
&\quad+ \frac{2}{p}\int_{\Omega_+\times \Omega_-}\frac{|s_1^{1/p} u^+(x)-s_2^{1/p} u^-(y)|^p}{|x-y|^{N+ps}}\, dx\, dy.
\end{split}
\end{equation}
and
\begin{equation}
\label{nablapsi}
\begin{split}
\frac{\partial\psi}{\partial s_1}(s_1, s_2)&=\frac{s_1^{(1-p)/p}}{p}\langle \J'(s_1^{1/p}u^++s_2^{1/p}u^-), u^+\rangle,\\
\frac{\partial\psi}{\partial s_2}(s_1, s_2)&=\frac{s_2^{(1-p)/p}}{p}\langle \J'(s_1^{1/p}u^++s_2^{1/p}u^-), u^-\rangle.
\end{split}
\end{equation}
The terms of $\psi$ in \eqref{pas} before the integral adds up to a strictly concave function, since $\min\{r, q\}> p$. Moreover for any $a, b\ge 0$, the function
\[
\R^2_+\ \ni (s_1, s_2)\mapsto |s_1^{1/p}a+s_2^{1/p}b|^p
\]
is also concave, which implies by integration that the last integral in \eqref{pas} is also a concave function of $(s_1, s_2)\in \R^2_+$. Therefore $\psi$
is strictly concave in $\R^2_+$. A direct computation shows that for $(s_1, s_2)\ne (0, 0)$
\[
\left.\frac{\partial}{\partial s_1 } \psi \, \right|_{s_2=0}=
\begin{cases}
\dfrac{[u^+]_{s,p}^p}{p}&\text{if $\min\{r, q\}>p$},\\[10pt]
\dfrac{[u^+]_{s,p}^p}{p}-\dfrac{\lambda}{p}\|u^+\|_p^p&\text{if $r=p<q$},\\[10pt]
\dfrac{[u^+]_{s,p}^p}{p}-\dfrac{\mu}{p}\left\|\dfrac{u^+}{|x|^{\alpha/p}}\right\|_p^p&\text{if $q=p<r$},
\end{cases}
\]
\[
 \left.\frac{\partial}{\partial s_2} \psi\, \right|_{s_1=0}=
\begin{cases}
\dfrac{[u^-]_{s,p}^p}{p}&\text{if $\min\{r, q\}>p$},\\[10pt]
\dfrac{[u^-]_{s,p}^p}{p}-\dfrac{\lambda}{p}\|u^-\|_p^p&\text{if $r=p<q$},\\[10pt]
\dfrac{[u^-]_{s,p}^p}{p}-\dfrac{\mu}{p}\left\|\dfrac{u^-}{|x|^{\alpha/p}}\right\|_p^p&\text{if $q=p<r$},
\end{cases}
\]
In both cases the derivatives are strictly positive due to \eqref{assumptions}, hence the maximum of $\psi$ is attained in the interior of $\R^2_+$ and $\psi$
 has its (unique) maximum at $(\bar s_1, \bar s_2)$ if and only if $\nabla \psi(\bar s_1, \bar s_2)=(0, 0)$, which corresponds to the unique maximum for $\J(t_1u^++t_2u^-)$ setting $\bar t_i=\bar s_i^{1/p}$, $i=1, 2$. Explicitly computing $\nabla \psi(\bar s_1, \bar s_2)=0$ through \eqref{nablapsi} gives conditions \eqref{maxcond}.

To prove \eqref{tmon}, observe that the concavity of $\psi$ is equivalent to
\[
(\nabla \psi (\xi)-\nabla \psi (\eta))\cdot(\eta-\xi)\geq 0, \qquad \forall \xi, \eta\in \R_+^2.
\]
 Letting $\xi=(1, 1)$ and alternatively $\eta=(\bar s_1, 1)$ or $\eta=(1, \bar s_2)$, where $(\bar s_1, \bar s_2)$ is the maximum point for $\psi$, we obtain
\[
\frac{\partial \psi}{\partial s_1}(1, 1)(\bar s_1-1)\geq 0, \qquad \frac{\partial \psi}{\partial s_2}(1, 1)(\bar s_2-1)\geq 0.
\]
Using \eqref{nablapsi} in these relations gives \eqref{tmon}, since $\bar s_i-1$ and $\bar s_i^{1/p}-1$ have the same sign.
\end{proof}

\begin{theorem}[Subcritical case]\label{subc}
Let \eqref{assumptions} be fulfilled with $q<p^*_\alpha$ (thus necessarily $\alpha<ps$). Then the problem
\begin{equation}
\label{c2}
c_2:=\inf_{u\in {\mathcal N}_{\rm sc}}\J(u)=\inf_{u^\pm\neq 0} \sup_{(t_1, t_2)\in \R^2} \J(t_1 u^+ +t_2 u^-)
\end{equation}
has a solution $v$ which is a sign-changing critical point for $\J$ of minimal energy.
\end{theorem}

\begin{proof}
The second equality in \eqref{c2} follows from the previous Lemma. Since ${\mathcal N}_{\rm sc}\subseteq {\mathcal N}$ we immediately infer $c_2\geq c_1>0$. Due to Lemma \ref{25}
\[
\langle \J'(u), u^\pm\rangle\ge \langle \J'(u^\pm), u^\pm\rangle, \qquad \forall u\in W^{s,p}_0(\Omega),
\]
so we deduce
\[
\langle \J'(u^\pm), u^\pm\rangle\le 0,\quad \forall  u\in {\mathcal N}_{\rm sc}
\]
and Lemma \ref{lemmalb} provides $\delta_0$ depending only on the parameters such that
\begin{equation}
\label{kl2}
[u^\pm]_{s,p}\ge \delta_0, \qquad \forall  u\in {\mathcal N}_{\rm sc}.
\end{equation}
Pick a minimizing sequence $\{v_n\}_n\subseteq {\mathcal N}_{\rm sc}$.
Since $\langle \J'(v_n), v_n\rangle=0$ and $\J(v_n)\to c_2$, Remark \ref{gps} ensures that $\{v_n\}_n$ is bounded and, up to subsequences, converges weakly in $W^{s,p}_0(\Omega)$ and strongly in $L^r(\Omega)$ and $L^q(\Omega, dx/|x|^\alpha)$ to some $v$. First observe that $v^\pm\neq 0$. Indeed we have
\begin{equation}
\label{gkj}
\int_\Omega \frac{|v_n^\pm|^q}{|x|^\alpha}\, dx\to \int_\Omega \frac{|v^\pm|^q}{|x|^\alpha}\, dx,\quad \int_\Omega|v_n^\pm|^{r} \, dx\to \int_\Omega|v^\pm|^{r}\, dx
\end{equation}
and since  $\langle \J'(v_n), v_n^\pm \rangle=0$ for all $n\in \mathbb{N}$, we deduce from \eqref{b} and \eqref{kl2}
\[
\begin{split}
\int_{\Omega} \lambda |v^\pm|^r+\mu\frac{|v^\pm|^q}{|x|^\alpha}\, dx&=\lim_n\int_{\Omega} \lambda |v_n^\pm|^r+\mu\frac{|v_n^\pm|^q}{|x|^\alpha}\, dx\\
&=\lim_n\langle (-\Delta_p)^s v_n, v_n^\pm\rangle\geq  \liminf_n[v_n^\pm]_{s,p}^p\geq \delta_0^p>0.
\end{split}
\]
Let us prove that $v\in {\mathcal N}_{\rm sc}$, i.e.
\begin{equation}
\label{hlj}
\langle \J'(v), v^\pm\rangle=0.
\end{equation}
The functional $u\mapsto  \langle (-\Delta_p)^s u, u^\pm\rangle$
is weakly sequentially lower semicontinuous by Fatou's lemma, since it can be represents as a nonnegative integral of the form $f(x, y, u(x), u(y))$. Therefore, also by \eqref{gkj},
$u \mapsto \langle \J'(u), u^\pm\rangle $
is weakly sequentially lower semicontinuous, and since \eqref{hlj} holds for any $v_n$, we deduce $\langle \J'(v), v^\pm\rangle\leq 0$.
Suppose that, say, $\langle \J'(v), v^+\rangle<0$ and let $\Phi(v)$ be the projection on ${\mathcal N}_{\rm sc}$ of $v$ given by  Lemma \ref{le51}, i.e.
\[
\Phi(v)=\bar t_+ v^++\bar t_-v^-, \qquad (\bar t_+, \bar t_-)={\rm Argmax}\J(t_1 v^++t_2v^-).
\]
 Then by \eqref{maxcond}, \eqref{tmon}, $\langle \J'(v), v^+\rangle< 0$ and $\langle \J'(v), v^-\rangle\leq 0$, it holds $\bar t_+<1$ and $\bar t_-\leq 1$. Since $\langle \J'(\Phi(v)), \Phi(v)\rangle=0$ we have 
\[
\begin{split}
\J(\Phi(v))&=\left(\frac{\mu}{p}-\frac{\mu}{q}\right)\int_\Omega\frac{|\Phi(v)|^q}{|x|^\alpha}\, dx+\left(\frac{\lambda}{p}-\frac{\lambda}{r}\right)\int_\Omega|\Phi(v)|^r\, dx\\
&= \left(\frac{\mu}{p}-\frac{\mu}{q}\right)\left(\bar{t}_+^q \int_\Omega\frac{|v^+|^q}{|x|^\alpha}\, dx+\bar t_-^q\int_{\Omega}\frac{|v^-|^q}{|x|^\alpha}\, dx\right)+\left(\frac{\lambda}{p}-\frac{\lambda}{r}\right)\left(\bar t_+^r\int_\Omega |v^+|^r\, dx+\bar t_-^r\int_\Omega |v^-|^r\, dx\right)  \\
&<\left(\frac{\mu}{p}-\frac{\mu}{q}\right)\int_\Omega\frac{|v|^q}{|x|^\alpha}\, dx+\left(\frac{\lambda}{p}-\frac{\lambda}{r}\right)\int_\Omega|v|^r\, dx\\
&=\lim_n\left(\frac{\mu}{p}-\frac{\mu}{q}\right)\int_\Omega\frac{|v_n|^q}{|x|^\alpha}\, dx+\left(\frac{\lambda}{p}-\frac{\lambda}{r}\right)\int_\Omega|v_n|^r\, dx=\lim_n \J(v_n)=c_2,
\end{split}
\]
which is a contradiction and proves \eqref{hlj}. This in turn implies that $v\in {\mathcal N}$ and then $v_n\to v$ strongly in $W^{s,p}_0(\Omega)$ due to
\[
\lim_n[v_n]_{s,p}^p=\lim_n\mu\int_\Omega\frac{|v_n|^q}{|x|^\alpha}\, dx+\lambda\int_\Omega|v_n|^r\, dx=\mu\int_\Omega\frac{|v|^q}{|x|^\alpha}\, dx+\lambda\int_\Omega|v|^r\, dx=[v]_{s,p}^p.
\]
and therefore $\J(v_n)\to \J(v)=c_2$, proving that $v$ solves problem \eqref{c2}.
Next we prove that $v$ is a critical point for $\J$.
Suppose by contradiction that $\J'(v)\ne 0$, then there exists $\varphi\in W^{s,p}_0(\Omega)$ such that $\langle \J'(v), \varphi\rangle<-1$ and therefore by continuity there exists a sufficiently small $\delta_0>0$ such that
\begin{equation}
\label{op}
\langle \J'(t_1 v^++t_2 v^-+\delta\varphi), \varphi\rangle<-1, \qquad \text{if }\ |1-t_1|, |1-t_2|, |\delta|<\delta_0.
\end{equation}
The function $\psi$  in Lemma \ref{le51} is smooth and strictly concave and has a strict maximum in $(1, 1)$, therefore for some $\eps>0$
\[
\nabla \psi(s_1, s_2)\ne (0,0) \qquad \forall 0<|(s_1-1, s_2-1)|\le\eps.
\]
This implies by changing variables $(s_1, s_2)\mapsto (t_1^p, t_2^p)$  that for some $\eps'>0$, which we can suppose smaller than $\delta_0/2$
\begin{equation}
\label{se}
\left(\langle \J'(t_1 v^++t_2 v^-), v^+\rangle, \langle \J'(t_1 v^++t_2 v^-), v^-\rangle\right)\ne (0, 0),\quad \text{if }\  0<|(t_{1}-1, t_2-1)|\le \eps'.
\end{equation}
Let $B=B((1, 1), \eps')$ and for any $\delta>0$, $(t_1, t_2)\in B$ define
\[
u=u(\delta, t_1, t_2)=t_1 v^++t_2 v^-+\delta \varphi,
\]
(which is continuous in $\delta, t_1, t_2$), noting that for sufficiently small $\delta$, $u^\pm\ne 0$. Now consider the field
\[
G_\delta(t_1, t_2)=\left(\langle \J'(u), u^+\rangle, \langle \J'(u), u^-\rangle\right)\in \R^2.
\]
By \eqref{se} it holds $\inf_{\partial B}|G_\delta|>0$ for $\delta=0$ and therefore by continuity the same holds for any sufficiently small $\delta>0$.
Clearly $\delta\mapsto G_\delta$ is a homotopy, and $G_0(1, 1)=(0,0)$. By elementary degree theory, the equation $G_\delta(t_1, t_2)=(0, 0)$ has a solution $(\tilde t_1,\tilde t_2)\in B$ for some small $\delta\in \ ]0, \delta_0/2[$. If $\bar v=\tilde t_1 v^++\tilde t_2 v^-+\delta\varphi$, this amounts to $\bar v\in {\mathcal N}_{\rm sc}$. Since it holds
\[
\J(\bar v)=\J(\tilde t_1 v^++\tilde t_2 v^-)+\int_0^\delta \langle \J'( \tilde t_1 v^++\tilde t_2 v^- +t\varphi), \varphi\rangle\, dt,
\]
\eqref{op} applied to the integrand provides
\[
\J(\bar v)\le \J(\tilde t_1 v^++\tilde t_2 v^-)-\delta.
\]
Since $v\in {\mathcal N}_{\rm sc}$ Lemma \ref{le51} gives $\J(\tilde t_1 v^++\tilde t_2 v^-)\leq \J(v)$, which, inserted into the previous inequality, contradicts the minimality of $\J(v)$.

\subsection{Asymptotics for the sign-changing level}

We now consider the sign-changing critical levels for the quasicritcal approximation, hence again assuming $\alpha<ps$ and \eqref{assumptionsb}. As in Section \ref{assection1}, for any small $\rho>0$ we let $q_\rho=p^*_\alpha-\rho$, and define the functional $J_\rho$ as per \eqref{Jrho}.
Let moreover ${\mathcal N}^\rho$ and ${\mathcal N}^\rho_{\rm sc}$ be the corresponding Nehari and sign-changing Nehari (see \eqref{Nsc}) manifolds, with  $w_\rho$ and $v_\rho$ being the nonnegative and sign-changing solutions of minimal energy. Finally, we let 
\[
c_{1, \rho}=\inf_{u\in {\mathcal N}^\rho} J_\rho  (u)=J_\rho(w_{\rho}),\qquad c_{2, \rho}=\inf_{u\in {\mathcal N}^\rho_{\rm sc}} J_\rho(u)=J_\rho(v_\rho).
\]
Since $0\in \Omega$, we choose  $\delta>0$ so small that $B_{5\theta\delta}\subseteq \Omega$, with $\theta$ given in Lemma \ref{leesti}. we consider a $C^\infty$, nonnegative truncation
\[
\eta(x)=
\begin{cases}
1&\text{if $|x|\ge 3\theta$},\\
0&\text{if $|x|\le 2\theta$},
\end{cases}
\]
letting furthermore
\begin{equation}
\label{wrhodelta}
\eta_\delta(x):=\eta\left(\frac{x}{\delta}\right),\qquad w_{ \rho, \delta}:=\eta_\delta w_\rho.
\end{equation}
\begin{lemma}
Assume \eqref{assumptionsb} and \eqref{condr}. Then there exist $\rho_0>0$ and $C$ independent of $\rho$  such that for any $\rho\in \ ]0, \rho_0]$ it holds
\begin{equation}
\label{vdelta}
[w_{\rho, \delta}]_{s,p}^p\le [w_{\rho}]_{s, p}^p+C\delta^{N-ps}.
\end{equation}
and
\[
\|w_{\rho, \delta}\|_{L^r}^r\ge \|w_{\rho}\|_{L^r}^r-C\delta^N.
\]
\end{lemma}

\begin{proof}
The second estimate trivially follows from the equi-boundedness of $\{w_\rho\}_\rho$, hence we focus on \eqref{vdelta}. Using Theorem \ref{equibound}, we observe that $w_{\rho}$ solves $(-\Delta_p)^s w_{\rho}=f_\rho$ for a suitable $f_\rho$ satisfying
\[
|f_\rho(x)|\leq C(1+|x|^{-\alpha})
\]
with a constant $C$ independent of $\rho$ for small $\rho>0$. Since we are assuming $\alpha<ps$, we can  fix $\bar q\in\  ]\frac{N}{ps}, \frac{N}{\alpha}[$ and deduce that $\|f_\rho\|_{L^{\bar q}}\leq C$ for a constant independent of $\rho$ for small $\rho>0$.
We have
\[
\begin{split}
[w_\rho \eta_\delta]_{s,p}^p &\le \int_{A_1} \frac{|w_\rho(x) - w_\rho(y)|^p}{|x - y|^{N+sp}}\, dx\, dy + \int_{A_2} \frac{|w_\rho(x)\, \eta_\delta(x) - w_\rho(y)\, \eta_\delta(y)|^p}{|x - y|^{N+sp}}\, dx\, dy  \\
& + 2 \int_{A_3} \frac{|w_\rho(x)\, \eta_\delta(x) - w_\rho(y)|^p}{|x - y|^{N+sp}}\, dx\, dy =: I_1 + I_2 + 2 I_3,
\end{split}
\]
where
\[
A_1 = B_{3 \theta \delta}^c \times B_{3 \theta \delta}^c, \quad
A_2 = B_{4 \theta \delta} \times B_{4 \theta \delta},\quad
A_3 = B_{3 \theta \delta} \times B_{4 \theta \delta}^c.
\]
Clearly, $I_1 \le [w_\rho]_{s,p}^p$. To estimate $I_2$, let $\varphi \in C^\infty_0(B_{5 \theta},[0,1])$ with $\varphi = \eta$ in $B_{4 \theta}$ and let $\varphi_\delta(x) = \varphi(x/\delta)$.
By \cite[Lemma 2.5]{mpsy} it holds
\[
[w_\rho\varphi_{\delta}]_{s,p}^p\leq C\|f_\rho\|_{L^{\bar q}}^{\frac{p}{p-1}} \left(\|\varphi_\delta\|_{L^{p\bar q'}}^p+[\varphi_\delta]_{s,p}^p\right)
\]
and by scaling
\[
\|\varphi_\delta\|_{L^{p\bar q'}}^p=\delta^{N/\bar q'}\|\varphi_1\|_{L^{p\bar q'}}^p,\qquad [\varphi_\delta]_{s,p}^p=\delta^{N-ps}[\varphi_1]_{s,p}^p,
\]
so that, being $\delta<1$ and $N/\bar q'>N-ps$, we obtain
\[
[w_\rho\varphi_{\delta}]_{s,p}^p\leq C\delta^{N-ps}
\]
for some $C$ independent of $\rho$ for small $\rho>0$.
Then, with the same constant
\[
I_2 = \int_{A_2} \frac{|w_\rho(x)\, \varphi_\delta(x) - w_\rho(y)\, \varphi_\delta(y)|^p}{|x - y|^{N+sp}}\, dx\, dy \le [w_\rho \varphi_\delta]_{s,p}^p \le C \delta^{N-ps}.
\]
 Finally, we employ the uniform bound on $w_\rho$ and  $|x - y| \ge |y| - 3 \theta \delta \ge |y|/4$ on $A_3$ to get
\[
I_3 \le C \|w_\rho\|_{\infty}^p \int_{A_3} \frac{dx dy}{|y|^{N+sp}} \le C \delta^{N-sp}.
\]
This concludes the proof.
\end{proof}

\begin{lemma}
	\label{energy-control}
Let \eqref{assumptionsb} be fulfilled and moreover $r>p^*-1$.
 Then, there exists  $\kappa>0$  such that for any sufficiently small $\rho>0$ it holds
\begin{equation}
\label{estc2rho}
c_{2,\rho}\le c_{1, \rho}+\left(\frac 1 p -\frac{1}{p^*_\alpha}\right)\frac{S_\alpha^{\frac{N-\alpha}{ps-\alpha}}}{\mu^{\frac{p}{p^*_\alpha-p}}}  -\kappa .
\end{equation}
\end{lemma}

\begin{proof}
We first observe that the condition $r>p^*-1$ implies \eqref{condr} in the range $r\geq p$, which we are assuming. Moreover, since $p^*-1>p^*/p'$, we can suppose that $r> p^*/p'$
always holds true.
For $\eps<\delta<{\rm dist}(0, \partial\Omega)/5\theta$ to be chosen later consider $v=w_{\rho, \delta}- u_{\eps, \delta}$ where we set $u_{\eps, \delta}=u_{\alpha, \eps, \delta}$ as defined in \eqref{uepsdeltadef} for simplicity and $w_{\rho, \delta}$ is defined in \eqref{wrhodelta}. Clearly $v$ is sign changing with $v^+=w_{\rho, \delta}$ and $v^-=-u_{\eps, \delta}$. By construction we have
\[
c_{2, \rho}\le \sup_{ (t_1, t_2)\in \R^2_+} J_\rho(t_1 w_{\rho, \delta}-t_2 u_{\eps, \delta}).
\]
We start estimating the Gagliardo norm observing that
\begin{equation}
\label{supports}
{\rm supp}(w_{\rho, \delta})\subseteq B_{2\theta\delta}^c,\quad {\rm supp}(u_{\eps, \delta})\subseteq B_{\theta\delta},
\end{equation}
so that
\begin{equation}
\label{pqs}
\begin{split}
[t_1 w_{\rho, \delta}-t_2 u_{\eps, \delta}]_{s,p}^p\leq&
t_1^p\int_{A_1}\frac{|w_{\rho, \delta}(x)-w_{\rho, \delta}(y)|^p}{|x-y|^{N+ps}}\, dx\, dy +
t_2^p\int_{A_2}\frac{|u_{\eps, \delta}(x)-u_{\eps, \delta}(y)|^p}{|x-y|^{N+ps}}\, dx\, dy\\
&\ +
2\int_{A_3}\frac{|t_1w_{\rho, \delta}(x)+t_2u_{\eps, \delta}(y)|^p}{|x-y|^{N+ps}}\, dx\, dy
\end{split}
\end{equation}
where
\[
A_1=B_{\theta\delta}^c\times B_{\theta\delta}^c, \quad A_2=B_{\theta\delta}\times B_{\theta\delta},\quad A_3=B_{\theta\delta}^c\times B_{\theta\delta}.
\]
For the last integral in \eqref{pqs}, we use the following elementary inequalities:
\[
|a+b|^p\leq
\begin{cases}
 |a|^p+|b|^p +C(|a|^{p-1}|b|+|a||b|^{p-1})&\text{if $p\geq 2$},\\
 |a|^p+|b|^p +C|a||b|^{p-1}&\text{if $1<p<2$},
 \end{cases}
\]
with $a=t_1w_{\rho, \delta}(x)$, $b=t_2u_{\eps, \delta}(y)$. Since $w_{\rho, \delta}(x)=u_{\eps, \delta}(y)=0$ for $(x, y)\in A_3$, this gives
\[
\begin{split}
 \int_{A_3}\frac{|t_1w_{\rho, \delta}(x)+t_2u_{\eps, \delta}(y)|^p}{|x-y|^{N+ps}}\, dx\, dy&\leq t_1^p\int_{A_3}\frac{|w_{\rho, \delta}(x)-w_{\rho, \delta}(y)|^p}{|x-y|^{N+ps}}\, dx\, dy+t_2^p\int_{A_3}\frac{|u_{\eps, \delta}(x)-u_{\eps, \delta}(y)|^p}{|x-y|^{N+ps}}\, dx\, dy\\
 &\ +C
 \begin{cases}
 t_1^{p-1}t_2I_1(w_{\rho, \delta}, u_{\eps, \delta})+t_1t_2^{p-1}I_{p-1}(w_{\rho, \delta}, u_{\eps, \delta})&\text{if $p> 2$},\\
t_1t_2^{p-1}I_{p-1}(w_{\rho, \delta}, u_{\eps, \delta}) &\text{if $1<p\leq2$},
\end{cases}
\end{split}
 \]
 where, recalling \eqref{supports},
 \[
I_q(w_{\rho, \delta}, u_{\eps, \delta})=\int_{B_{2\theta\delta}^c\times B_{\theta\delta}}\frac{|w_{\rho, \delta}(x)|^{p-q}|u_{\eps, \delta}(y)|^q}{|x-y|^{N+ps}}\, dx\, dy.
\]
We reassemble the integrals over the $A_i$'s to obtain
\[
[t_1 w_{\rho, \delta}-t_2 u_{\eps, \delta}]_{s,p}^p\le t_1^p[w_{\rho, \delta}]_{s, p}^p +t_2^p[u_{\eps, \delta}]_{s,p}^p+
C t_1^{p-1}t_2I_1(w_{\rho, \delta}, u_{\eps, \delta})+C(p-2)^+t_1t_2^{p-1}I_{p-1}(w_{\rho, \delta}, u_{\eps, \delta})
\]
and then we proceed estimating the last two terms.
Since both \eqref{assumptionsb} and \eqref{condr} hold, Theorem \eqref{equibound} gives a uniform $L^\infty$ bound on $w_{\rho}$. The latter, together with the inequality  $|x-y|\geq |x|/2$ for $x\in B_{2\theta\delta}^c$ and $y\in B_{\theta\delta}$, implies
\[
I_q(w_{\rho, \delta}, u_{\eps, \delta})\leq C\int_{B_{2\theta\delta}^c}|x|^{-N-ps}\, dx\int_{\Omega} u_{\eps, \delta}^q(y)\, dy\leq C\delta^{-ps}\int_{\Omega} u_{\eps, \delta}^q(y)\, dy.
\]
Observe that $p-1<p^*/p'$, so that \eqref{setqab} for $\beta=p-1$ provides
\[
I_{p-1}(w_{\rho, \delta}, u_{\eps, \delta})\leq C\delta^{-ps} \eps^{\frac{N-ps}{p(p-1)}(p-1)}\delta^{N-\frac{N-ps}{p-1}(p-1)}=C\eps^{\frac{N-ps}{p}}.
\]
The term $I_1(w_{\rho, \delta}, u_{\eps, \delta})$ only appears if $p>2$, which in turn forces $p^*> p'$: in this case \eqref{setqab} for $\beta=1$ reads
\[
I_{1}(w_{\rho, \delta}, u_{\eps, \delta})\leq C\delta^{-ps} \eps^{\frac{N-ps}{p(p-1)}}\delta^{N-\frac{N-ps}{p-1}}=C \eps^{\frac{N-ps}{p(p-1)}}\delta^{\frac{N-ps}{p-1}(p-2)}.
\]
All in all, we get
\[
[t_1 w_{\rho, \delta}-t_2 u_{\eps, \delta}]_{s,p}^p\le t_1^p[w_{\rho, \delta}]_{s, p}^p +t_2^p[u_{\eps, \delta}]_{s,p}^p+ Ct_1t_2^{p-1}\eps^{\frac{N-ps}{p}}+ C(p-2)^+ t_1^{p-1}t_2\eps^{\frac{N-ps}{p(p-1)}}\delta^{\frac{N-ps}{p-1}(p-2)}.
\]
We set $\delta=\eps^{1/p}$, apply \eqref{vdelta} on the first term and \eqref{estqa} on the second to obtain
\begin{equation}
\label{e1}
\begin{split}
[t_1 w_{\rho, \eps^{1/p}}-t_2 u_{\eps, \eps^{1/p}}]_{s,p}^p&\le t_1^p[w_\rho]_{s,p}^p+t_2^pS_\alpha^{\frac{N-\alpha}{ps-\alpha}}+t_1^p\delta^{N-ps}+t_2^p(\eps/\delta)^{\frac{N-ps}{p-1}}\\
&\quad +Ct_1t_2^{p-1}\eps^{\frac{N-ps}{p}}+ C(p-2)^+ t_1^{p-1}t_2\eps^{\frac{N-ps}{p(p-1)}}\delta^{\frac{N-ps}{p-1}(p-2)}\\
&\le t_1^p[w_\rho]_{s,p}^p+t_2^pS_\alpha^{\frac{N-\alpha}{ps-\alpha}}+C(t_1^p+t_2^p)\eps^{\frac{N-ps}{p}},
\end{split}
\end{equation}
where we used Young's inequality as $t_1t_2^{p-1}+t_1^{p-1}t_2\leq C(t_1^p+t_2^p)$.
Next we estimate the other terms. By the equi-boundedness of $\{w_{\rho}\}_\rho$  it follows
\[
\int_\Omega w_{\rho, \eps^{1/p}}^r\, dx\ge \int_\Omega w_{\rho}^r\, dx-C\eps^{ \frac{N}{p}},\qquad  \int_\Omega \frac{w_{\rho, \eps^{1/p}}^{p^*_\alpha-\rho}}{|x|^{\alpha}}\, dx\ge \int_\Omega \frac{w_{\rho}^{p^*_\alpha-\rho}}{|x|^{\alpha}}\, dx-C\eps^{\frac{N-\alpha}{p}},
\]
so that, since $\alpha<ps$ and $\eps<1$, it holds
\begin{equation*}
\int_\Omega w_{\rho, \eps^{1/p}}^r\, dx\ge \int_\Omega w_{\rho}^r\, dx-C\eps^{ \frac{N-ps}{p}},\qquad  \int_\Omega \frac{w_{\rho, \eps^{1/p}}^{p^*_\alpha-\rho}}{|x|^{\alpha}}\, dx\ge \int_\Omega \frac{w_{\rho}^{p^*_\alpha-\rho}}{|x|^{\alpha}}\, dx-C\eps^{\frac{N-ps}{p}}.
\end{equation*}
Moreover, for any $\eps>0$ we can choose $\rho_1(\eps)>0$ so that whenever $\rho<\rho_1(\eps)$, it holds
\begin{equation*}
\int_\Omega \frac{u^{p^*_\alpha-\rho}_{\eps,\eps^{1/p}}}{|x|^{\alpha}}\, dx \ge  \int_\Omega \frac{u_{\eps,{\eps^{1/p}}}^{p^*_\alpha}}{|x|^{\alpha}}\, dx-\eps^{\frac{N-ps}{p}}\underset{\eqref{estqaf}}{\geq}  S_\alpha^{\frac{N-\alpha}{ps-\alpha}}- C(\eps/\delta)^{\frac{N-ps}{p-1}}=S_\alpha^{\frac{N-\alpha}{ps-\alpha}}- C\eps^{\frac{N-ps}{p}}.
\end{equation*}
Finally, recalling that we are assuming $r> p^*/p'$ and using \eqref{estqal}, we have
\begin{equation}
\label{e4}
\int_\Omega u_{\eps, \eps^{1/p}}^r\, dx\ge c\, \eps^{N-\frac{N-ps}{p}r}.
\end{equation}
Gathering together \eqref{e1} -- \eqref{e4} we obtain for any $\rho<\rho_1(\eps)$
\[
J_\rho(t_1 w_{\rho, \eps^{1/p}}-t_2 u_{\eps, \eps^{1/p}})\le g_{\rho, \eps}(t_1)+h_{\rho, \eps}(t_2) 
\]
where
\[
\begin{split}
g_{\rho, \eps}(t_1)&:=J_\rho(t_1 w_{\rho})+ C\left(t_1^p+t_1^r+t_1^{p^*_\alpha-\rho}\right)\eps^{\frac{N-ps}{p}},\\
 h_{\rho, \eps}(t_2)&:=S_\alpha^{\frac{N-\alpha}{ps-\alpha}}\left(\frac{t_2^p}{p} - \mu \frac{t_2^{p^*_\alpha-\rho}}{p^*_\alpha-\rho}\right) -c\, \frac{ t_2^r}{r}\eps^{N-\frac{N-ps}{p}r}+C\left(t_2^p+t_2^{p^*_\alpha-\rho}\right)\eps^{\frac{N-ps}{p}}.
 \end{split}
\]
By Theorem \ref{equibound} again, both the zero-th order terms in $J_\rho(w_\rho)$ are uniformly bounded from below as $\rho\to 0$, so that for sufficiently small $\eps>0$ and $\rho<\rho_1(\eps)$ it holds
\[
\sup_{t_1\geq 0} g_{\rho, \eps}(t_1)\leq \sup_{t_1\geq 0} J_{\rho}(t_1w_\rho)+C\eps^{\frac{N-ps}{p}}\leq c_{1, \rho} +C\eps^{\frac{N-ps}{p}}
\]
with $C$ independent of $\rho$ and $\eps$. Moreover, elementary arguments show that for sufficiently small $\eps>0$, there exists $0<a<1<A<+\infty$ independent of $\rho<\rho_1(\eps)$, such that
\[
\begin{split}
\sup_{t_2\geq 0} h_{\rho, \eps}(t_2)&=\sup_{t_2\in [a, A]} h_{\rho, \eps}\leq \sup_{t_2\geq 0}S_\alpha^{\frac{N-\alpha}{ps-\alpha}}\left(\frac{t_2^p}{p} - \mu \frac{t_2^{p^*_\alpha-\rho}}{p^*_\alpha-\rho}\right)-c \frac{a^r}{r}\eps^{N-\frac{N-ps}{p}r}+C\left(A^p+A^{p^*_\alpha}\right)\eps^{\frac{N-ps}{p}}\\
&\leq\left(\frac{1}{p}-\frac{1}{p^*_\alpha-\rho}\right)\frac{S_\alpha^{\frac{N-\alpha}{ps-\alpha}}}{\mu^{\frac{p}{p^*_\alpha-\rho-p}}}-c'\eps^{N-\frac{N-ps}{p}r}+C'\eps^{\frac{N-ps}{p}}.
\end{split}
\]
Observe that there exists $0<\rho_2(\eps)<\rho_1(\eps)$ such that if $0<\rho<\rho_2(\eps)$ then
\[
\left(\frac{1}{p}-\frac{1}{p^*_\alpha-\rho}\right)\frac{S_\alpha^{\frac{N-\alpha}{ps-\alpha}}}{\mu^{\frac{p}{p^*_\alpha-\rho-p}}}\leq \left(\frac{1}{p}-\frac{1}{p^*_\alpha}\right)\frac{S_\alpha^{\frac{N-\alpha}{ps-\alpha}}}{\mu^{\frac{p}{p^*_\alpha-p}}} + \eps^{\frac{N-ps}{p}}
\]
Therefore, for any sufficiently small $\eps>0$ and any $0<\rho<\rho_2(\eps)$ it holds
\[
\sup_{(t_1, t_2)\in \R^2_+}J_\rho(t_1 w_{\rho, \eps^{1/p}}-t_2 u_{\eps, \eps^{1/p}})\leq 
c_{1, \rho}+\left(\frac{1}{p}-\frac{1}{p^*_\alpha}\right)\frac{S_\alpha^{\frac{N-\alpha}{ps-\alpha}}}{\mu^{\frac{p}{p^*_\alpha-p}}} + C'\eps^{\frac{N-ps}{p}}- c'\eps^{N-\frac{N-ps}{p}r}
\]
Since 
\[
r>p^*-1\quad \Leftrightarrow\quad N-\frac{N-ps}{p}r<\frac{N-ps}{p},
\]
we can choose a sufficiently small $\eps>0$ such that 
\[
c'\eps^{N-\frac{N-ps}{p}r}-C'\eps^{\frac{N-ps}{p}} =:\kappa>0,
\]
obtaining the claim.
\end{proof}

\subsection{The critical case}
We will need the following variant of the Concentration-Compactness principle.
Define for any $u\in L^1_{\rm loc}(\R^N)$ and a.e. Lebesgue point $x$ for $u$, the function
\[
|D^su|^p(x):=\int_{\R^N}\frac{|u(x)-u(y)|^p}{|x-y|^{N+ps}}\, dy.
\]
In the following, by a {\em measure} in $\R^N$ we mean a continuous linear functional on $C_0(\R^N)$, often called (finite) {\em Radon} measure. As such, any measure also define a continuous linear functional on $C_b(\R^N)$, the Banach space of bounded continuous functions with the $\sup$ norm. We will employ the weak$^*$ convergence notion with respect to this latter duality, which is often called {\em tight convergence}; it is a stronger notion of convergence than the usual weak$^*$ convergence of Radon measures (i.e. the one given by duality with $C_0(\R^N)$) and a tight compactness criterion does not follows from the usual Banach-Alaoglu theorem, but instead through Prokhorov's theorem, which says that a sequence $\{\mu_n\}_n$ is sequentially weak$^*$-compact (in duality with $C_b(\R^N)$) if and only if it is bounded and {\em tight}, i.e.
\[
\forall \eps>0\quad \exists K_\eps\Subset\R^N:\quad \sup_n\mu_n(\R^N\setminus K_\eps)<\eps.
\]

\begin{lemma}[Concentration Compactness with variable exponents]
\label{cc}
Let $0\leq \alpha\leq ps<N$, $\Omega$ be a bounded open subset of $\R^N$ with $0\in \Omega$ and $u_n\rightharpoonup u$ in
$W^{s,p}_0(\Omega)$. Given $q_n=p^*_\alpha-\rho_n$ for some $\rho_n\geq 0$, $\rho_n\to 0$, there exist two measures $\nu, \sigma$ and an at most countable set $\{x_j\}_{j\in {\mathcal J}}\subseteq\overline{\Omega}$  such that, up to subsequences,
\begin{align}
\label{meas}
&|D^s u_n|^p \weakto^* \sigma,\qquad \frac{|u_n|^{q_n}}{|x|^\alpha} \weakto^* \nu  \\
 &\sigma \geq  |D^s u|^p+\sum_{j\in {\mathcal J}} \sigma_j \delta_{x_j}, \qquad \sigma_j:=\sigma(\{x_j\}), \notag \\
 \label{nu}
 &\nu = \frac{|u|^{p^*_\alpha}}{|x|^\alpha}+\sum_{j\in {\mathcal J}} \nu_j \delta_{x_j}  , \qquad \nu_j:=\nu(\{x_j\}),\\
 \label{sobL}
 &\sigma_j\geq S_\alpha \nu_j^{\frac{p}{p^*_\alpha}},\qquad \forall j\in {\mathcal J}.
\end{align}
Moreover, if $\alpha>0$, then $\{x_j\}_{j\in {\mathcal J}}=\{0\}$. 
\end{lemma}

\begin{proof}
Clearly $\{|D^s u_n|^p\}_n$ is a tight sequence of bounded measures (tightness follows as in \cite[proof of Theorem 2.5]{MS}). H\"older and  Hardy-Sobolev inequalities readily give the boundedness of the sequence of measures $\{|u_n|^{q_n}/|x|^\alpha\}_n$ so that, up to subsequences, we can suppose that weak$^*$ convergence occurs to some $\sigma$ and $\nu$ respectively. Consider first the case $u=0$ and choose any $\varphi\in C^\infty_c(\R^N)$, $\varphi\geq 0$. Again H\"older and Hardy-Sobolev inequalities provide for any $\theta>0$
\begin{equation}
\label{ppp}
\begin{split}
S_\alpha\left(\int_{{\rm supp}(\varphi)}\frac{dx}{|x|^\alpha}\right)^{\frac{p}{p^*_\alpha}-\frac{p}{q_n}}&\left(\int_{\R^N} \varphi^{q_n}\frac{|u_n|^{q_n}}{|x|^\alpha}\, dx\right)^{\frac{p}{q_n}}\leq S_\alpha\left(\int_\Omega\frac{|\varphi u_n|^{p^*_\alpha}}{|x|^\alpha}\, dx\right)^{\frac{p}{p^*_\alpha}}\leq [\varphi u_n]_{s,p}^p\\
&\leq (1+\theta)\int_{\R^N}\varphi^p|D^s u_n|^p\, dx+ C_\theta\int_{\R^N}|D^s\varphi|^p|u_n|^p\, dx.
\end{split}
\end{equation}
(see \cite[proof of Theorem 2.5]{MS}). For the left hand side observe that  since $\varphi^{q_n}\to \varphi^{p^*_\alpha}$ uniformly, it holds
\[
\int_{\R^N} \varphi^{q_n}\frac{|u_n|^{q_n}}{|x|^\alpha}\, dx\to \int_{\R^N}\varphi^{p^*_\alpha}\, d\nu.
\]
If $\alpha>0$, then $p^*_\alpha<p^*$ and if $0\notin {\rm supp}(\varphi)$, the left hand side vanishes due to the strong convergence of $\{u_n\}_n$ in $L^{p^*_\alpha}_{\rm loc}(\R^N\setminus \{0\})$. This implies that  
\begin{equation}
\label{alpha>0}
\alpha>0\quad \Rightarrow\quad {\rm supp}(\nu)\subseteq \{0\},
\end{equation}
and in this case we set $\{x_j\}_{j\in {\mathcal J}}=\{0\}$.
Since $u_n\to 0$ strongly in $L^p(\R^N)$  and $|D^s\varphi|^p\in L^\infty(\R^N)$ by \cite[Lemma 2.3]{mm}, the last term on the right hand side of \eqref{ppp} vanishes as $n\to +\infty$, so that
\[
S_\alpha\left(\int_{\R^N}\varphi^{p^*_\alpha}\, d\nu\right)^{\frac{1}{p^*_\alpha}}\leq (1+\theta) \left(\int_{\R^N} \varphi^p\, d\sigma\right)^{\frac{1}{p}}.
\]
The estimates \eqref{meas}--\eqref{sobL} now follow letting $\theta\to 0$ and applying \cite[Lemma 1.2]{Lions} in the case $\alpha=0$ (the case $\alpha>0$ is actually easier due to \eqref{alpha>0}). In the case $u\neq 0$ we observe that, for all $\varphi\in C^0_b(\R^N)$, Lemma \ref{blle} implies
\begin{equation}
\label{brezislieb}
\lim_{n\to +\infty} \int_{\R^N}\frac{|u_n|^{q_n}}{|x|^\alpha}\,\varphi\,  dx -\int_{\R^N}\frac{|u_n-u|^{q_n}}{|x|^\alpha}\, \varphi\, dx=\int_{\R^N}\frac{|u|^{p^*_\alpha}}{|x|^\alpha}\,\varphi\, dx,
\end{equation}
and proceeding as in the proof of Lemma \ref{colema1}, point (1), we have
\begin{equation}
\label{brli}
\lim_n\int_{\R^N}|D^s u_n|^p\varphi\, dx-\int_{\R^N}|D^s(u_n-u)|^p\varphi\, dx=\int_{\R^N}|D^s u|^p\varphi\, dx.
\end{equation}
Then we apply the previous case to $w_n=u_n-u\rightharpoonup 0$, letting $\nu$ and $\sigma$ be the corresponding measures. Clearly \eqref{brezislieb} immediately gives
\[
\frac{|u_n|^{q_n}}{|x|^\alpha}\weakto^* \nu + \frac{|u|^{p^*_\alpha}}{|x|^\alpha},
\]
while \eqref{brli}  gives
\[
|D^s u_n|^p\weakto^* \sigma + |D^s u|^p.
\]
 The claimed representation then follows.
\end{proof}

{\bf Notations}: Let $v_\rho$ be the sign-changing solution obtained at level $c_{2, \rho}$, with $q_\rho=p^*_\alpha-\rho$. Then \eqref{estc2rho}, \eqref{convc} and $c_{2, \rho}\geq 0$ ensure that the hypotheses of Lemma \ref{bsol} are fulfilled, and therefore for a suitable sequence $\rho_n\to 0$, $v_{\rho_n}=:v_n\weakto v$  in $W^{s,p}_0(\Omega)$ where $v$ solves the critical problem. We can apply the previous Lemma to $\{v_n\}_n$ and to both $\{v_n^\pm\}_n$,  so that, up to a not relabelled subsequence, \eqref{meas}--\eqref{sobL} hold true for  $v$ and some $\sigma, \nu$, $\{x_j\}_{j\in {\mathcal J}}$, and, correspondingly, for $v^\pm$ and suitable $\sigma^\pm$, $\nu^\pm$, $\{x_j^\pm\}_{j\in {\mathcal J}^\pm}$. For this sequence, set as usual $q_n=p^*_\alpha-\rho_n$.

\begin{lemma}
\label{lfin}
With the previous notations, suppose that $\nu_j^\pm=\nu^{\pm}(\{x_j\})\neq 0$ for some $j\in {\mathcal J}^\pm$. Then
\begin{equation}
\label{masse}
\nu_j^\pm\geq \left(\frac{S_\alpha}{\mu}\right)^{\frac{N-\alpha}{ps-\alpha}}.
\end{equation}
\end{lemma}

\begin{proof}
 Fix $j\in {\mathcal J}^\pm$ such that $\nu_j^\pm>0$, $x_j\in \overline\Omega$ and for any $\delta>0$, let $\varphi_\delta\in C^\infty_c(B_{2\delta}(x_j))$ satisfy
\[
0\leq \varphi_\delta\leq 1,\qquad \varphi\lfloor_{B_\delta(x_j)}=1,\qquad \|\nabla \varphi_\delta\|_\infty\leq C/\delta.
\]
We test the equation $J_{\rho_n}'(v_n)=0$ with $\varphi_\delta v_n^\pm\in W^{s,p}_0(\Omega)$, we get
\[
\begin{split}
\mu\int_\Omega&\frac{|v_n^\pm|^{q_n}}{|x|^\alpha}\varphi_\delta\, dx+\lambda\int_{\Omega} |v_n^\pm|^r\varphi_\delta\\
&=\int_{\R^{2N}}\frac{|v_n(x)-v_n(y)|^{p-2}(v_n(x)-v_n(y))(\varphi_\delta(x)v^\pm_n(x)-\varphi_\delta(y)v^\pm_n(y))}{|x-y|^{N+ps}}\, dx\, dy \\
&\geq \int_{\R^{N}}|D^sv^\pm_n|^{p}\varphi_\delta \, dx   -\left|\int_{\R^{2N}}\frac{|v_n(x)-v_n(y)|^{p-2}(v_n(x)-v_n(y))v^\pm_n(y)(\varphi_\delta(x)-\varphi_\delta(y))}{|x-y|^{N+ps}}\, dx\, dy\right|.
\end{split}
\]
Proceeding as in \cite[Lemma 3.1]{MS} we can pass to the limit in $n$ to obtain that
\[
\int_{\R^N}\varphi_\delta\, d\sigma\leq \mu\int_{\R^N}\varphi_\delta\, d\nu+C\Big(\int_{\R^{N}}|D^s\varphi_\delta|^p|v^\pm|^p\, dy\Big)^{\frac 1 p} +\lambda\int_{\Omega} |v^\pm|^r\varphi_\delta\, dx.
\]
From \cite[(2.14)]{MS} the last two terms on the right go to zero for $\delta\to 0$, giving $\mu\, \nu_j^\pm\geq \sigma(\{x_j\})=: \sigma_j^\pm$. This,  coupled together with \eqref{sobL}, gives \eqref{masse}.
\end{proof}

\vskip3pt
\noindent
$\bullet$ {\em Proof of Theorem \ref{sctheo}}:
Let $\rho_n\to 0$ be a sequence such that Lemmata \ref{cc} and \ref{lfin} hold for the sign-chaning solutions $v_n$ at level $c_{2, \rho_n}$, keeping the notations settled before Lemma \ref{lfin}. Since $v_n\in {\mathcal N}^{\rho_n}$ it holds
\[
J_{\rho_n}(v_n)=\left(\frac{\lambda}{p}-\frac{\lambda}{r}\right)\int_\Omega |v_n|^r\, dx +\left(\frac{\mu}{p}-\frac{\mu}{q_n}\right)\int_{\Omega}\frac{|v_n|^{q_n}}{|x|^\alpha}\, dx.
\]
By Lemma \ref{bsol}, $\{v_n\}_n$ is bounded and up to subsequences weakly converges to a solution $v\in W^{s,p}_0(\Omega)$ of the critical problem.
First observe that, through \eqref{estc2rho}, \eqref{convc}
\begin{equation}
\label{final}
\lim_{n} J_{\rho_n}(v_n)=\lim_{n}c_{2, \rho_n}\leq c_{1, 0}+\left(\frac 1 p -\frac{1}{p^*_\alpha}\right)\frac{S_\alpha^{\frac{N-\alpha}{ps-\alpha}}}{\mu^{\frac{p}{p^*_\alpha-p}}}-\kappa,
\end{equation}
for some $\kappa>0$.
We claim that $v$ is sign changing. Suppose not, say $v\leq 0$: then, without loss of generality, we can assume up to subsequences that $v_n^+\rightharpoonup 0$ in $W^{s,p}_0(\Omega)$, $v_n^+\to 0$ in $L^r(\Omega)$ and pointwise a.e..

Thanks  to $\langle J'(v_n), v_n^\pm\rangle=0$ and (\ref{b}), we infer $[v_n^\pm]_{s,p}^p\geq \delta>0$ from  Lemma \ref{lemmalb} for some $\delta$ independent of $n$. Then, with the same notations as in the previous Lemma, it cannot hold $\nu_j^+=\nu^+(\{x_j\})\equiv 0$  for all $j\in {\mathcal J}^+$, since otherwise \eqref{nu} would give
\[
\int_\Omega \frac{|v_n^+|^{q_n}}{|x|^\alpha}\, dx\to  \int_\Omega \frac{|v^+|^{p^*_\alpha}}{|x|^\alpha}\, dx=0
\]
 giving, up to subsequences, the contradiction
\[
\delta\leq \lim_{n}[v_n^+]_{s,p}^p\leq\lim_{n}\langle (-\Delta_p)^s v_n, v_n^+\rangle= \lim_{n} \lambda\int_\Omega |v_n^+|^r\, dx+\mu\int_\Omega \frac{|v_n^+|^{q_n}}{|x|^\alpha}\, dx=0
\]
(see \eqref{b} for the second inequality).
Then $\nu_j^+>0$ for some $j\in {\mathcal J}^+$ and \eqref{masse} implies
\[
\lim_{n} \left(\frac{\lambda}{p}-\frac{\lambda}{r}\right)\int_\Omega |v_n^+|^r\, dx +\left(\frac{\mu}{p}-\frac{\mu}{q_n}\right)\int_{\Omega}\frac{|v_n^+|^{q_n}}{|x|^\alpha}\, dx\geq \left(\frac \mu p -\frac{\mu}{p^*_\alpha}\right)\nu_j^+\geq  \left(\frac 1 p -\frac{1}{p^*_\alpha}\right)\frac{S_\alpha^{\frac{N-\alpha}{ps-\alpha}}}{\mu^{\frac{p}{p^*_\alpha-p}}}.
\]
Moreover,  the weak limit  of $v_n^-$ cannot be zero since otherwise we would deduce again $\nu_j^->0$ for some $j\in {\mathcal J}^-$ and
\[
\lim_{n} \left(\frac{\lambda}{p}-\frac{\lambda}{r}\right)\int_\Omega |v_n^-|^r\, dx +\left(\frac{\mu}{p}-\frac{\mu}{q_n}\right)\int_{\Omega}\frac{|v_n^-|^{q_n}}{|x|^\alpha}\, dx\geq \left(\frac 1 p -\frac{1}{p^*_\alpha}\right)\frac{S_\alpha^{\frac{N-\alpha}{ps-\alpha}}}{\mu^{\frac{p}{p^*_\alpha-p}}},
\]
which implies
\[
2\left(\frac 1 p -\frac{1}{p^*_\alpha}\right)\frac{S_\alpha^{\frac{N-\alpha}{ps-\alpha}}}{\mu^{\frac{p}{p^*_\alpha-p}}}\leq \lim_{n} \left(\frac{\lambda}{p}-\frac{\lambda}{r}\right)\int_\Omega |v_n|^r\, dx +\left(\frac{1}{p}-\frac{1}{q_n}\right)\int_{\Omega}\frac{|v_n|^{q_n}}{|x|^\alpha}\, dx=\lim_{n} J_{\rho_n}(v_n).
\]
This, together with \eqref{final}, contradicts \eqref{c1ps}.
Therefore $v=v^-$ is a nontrivial solution to the critical problem $J_0'(v)=0$ and by \eqref{neh} it satisfies $J_0(v)\geq c_{1, 0}$.
But then, proceeding as before we have
 \[
 \begin{split}
 \lim_{n} J_{\rho_n}(v_n)&=\lim_{n}\left(\frac{\lambda}{p}-\frac{\lambda}{r}\right)\int_\Omega |v_n|^r\, dx +\left(\frac{\mu}{p}-\frac{\mu}{q_n}\right)\int_{\Omega}\frac{|v_n|^{q_n}}{|x|^\alpha}\, dx\\
&= \left(\frac{\lambda}{p}-\frac{\lambda}{r}\right)\int_\Omega |v|^r\, dx +\left(\frac{\mu}{p}-\frac{\mu}{p^*_\alpha}\right)(\nu^++\nu^-)(\R^N)\\
&\underset{\eqref{nu}}{\geq} \left(\frac{\lambda}{p}-\frac{\lambda}{r}\right)\int_\Omega |v|^r\, dx +\left(\frac{\mu}{p}-\frac{\mu}{p^*_\alpha}\right)\int_{\Omega}\frac{|v^-|^{p^*_\alpha}}{|x|^\alpha}dx+\left(\frac{\mu}{p}-\frac{\mu}{p^*_\alpha}\right)\nu_j^+\\
& \underset{\eqref{masse}}{\geq} J_0(v)+\left(\frac{1}{p}-\frac{1}{p^*_\alpha}\right)\frac{S_\alpha^{\frac{N-\alpha}{ps-\alpha}}}{\mu^{\frac{p}{p^*_\alpha-p}}}\geq c_{1, 0}+\left(\frac{1}{p}-\frac{1}{p^*_\alpha}\right)\frac{S_\alpha^{\frac{N-\alpha}{ps-\alpha}}}{\mu^{\frac{p}{p^*_\alpha-p}}},
\end{split}
 \]
again contradicting \eqref{final}. Therefore $v$ is a nontrivial sign-changing solution, and it remains to prove that it minimizes the energy. To this end, fix arbitrarily $u\in {\mathcal N}^0_{\rm sc}$ and for $\rho_n\to 0$ such that $v_n\weakto v$, let $u_n=t_{1n} u^++t_{2n} u^-$ solve
 \[
  J_{\rho_n}(u_n)=\sup_{t_1, t_2\geq 0}J_{\rho_n}(t_1 u^++t_2 u^-),
 \]
so that $u_n\in {\mathcal N}_{\rm sc}^{\rho_n}$. Proceeding as in the proof of \eqref{convc}, we see that $(t_{1n}, t_{2n})\to (1, 1)$: indeed, the 
functions $-\psi_n(t_1, t_2)=-J_{\rho_n}(t_1 u^++t_2u^-)$ are strictly convex on $\R^2_+$ due to the proof of Lemma \ref{le51} and uniformly coercive, hence  
${\rm Argmin}(-\psi_n)=(t_{1n}, t_{2n})\to (1, 1)={\rm Argmin}(-\psi),$,  with $\psi(t_1, t_2)=J_0(t_1 u^++t_2u^-)$. 

This fact, together with 
\[
\begin{split}
J_{\rho_n}(u_n)&=\left(\frac{\lambda}{p}-\frac{\lambda}{r}\right)\int_\Omega |u_n|^r\, dx +\left(\frac{\mu}{p}-\frac{\mu}{q_n}\right)\int_{\Omega}\frac{|u_n|^{q_n}}{|x|^\alpha}\, dx\\
&=t_{1n}^r\left(\frac{\lambda}{p}-\frac{\lambda}{r}\right)\int_\Omega |u^+|^r\, dx + t_{1n}^{q_n}\left(\frac{\mu}{p}-\frac{\mu}{q_n}\right)\int_{\Omega}\frac{|u^+|^{q_n}}{|x|^\alpha}\, dx\\
&\quad +t_{2n}^r\left(\frac{\lambda}{p}-\frac{\lambda}{r}\right)\int_\Omega |u^-|^r\, dx +t_{2n}^{q_n}\left(\frac{\mu}{p}-\frac{\mu}{q_n}\right)\int_{\Omega}\frac{|u^-|^{q_n}}{|x|^\alpha}\, dx,
\end{split}
\]
implies that $J_{\rho_n}(u_n)\to J_0(u)$. On the other hand  Fatou's Lemma gives (up to subsequences)
\[
\begin{split}
J_0(v)&=\left(\frac{\lambda}{p}-\frac{\lambda}{r}\right)\int_\Omega |v|^r\, dx +\left(\frac{\mu}{p}-\frac{\mu}{q}\right)\int_{\Omega}\frac{|v|^{q}}{|x|^\alpha}\, dx\\
&\leq \lim_{n}\left(\frac{\lambda}{p}-\frac{\lambda}{r}\right)\int_\Omega |v_n|^r\, dx +\left(\frac{\mu}{p}-\frac{\mu}{q_n}\right)\int_{\Omega}\frac{|v_n|^{q_n}}{|x|^\alpha}\, dx=\lim_n J_{\rho_n}(v_n).
\end{split}
\]
Since $u_n\in {\mathcal N}^{\rho_n}_{\rm sc}$, we have $J_{\rho_n}(v_n)\leq J_{\rho_n}(u_n)$ by construction, therefore
\[
J_0(v)\leq \lim_n J_{\rho_n}(v_n)\leq \lim_nJ_{\rho_n}(u_n)=J_0(u)
\]
and the minimality of $v$ is proved.
\qed

In the case $\alpha=ps$ the subcritical approximation is not more available, but nevethelss if $\mu<\lambda_{1, ps}$ existence of a sign-changing solution follows through direct minimization over ${\mathcal N}_{{\rm sc}}$.

\vskip3pt
\noindent
$\bullet$ {\em Proof of Theorem \ref{sctheoH}}:
We proceed as in the proof of Theorem \ref{subc}, picking a minimizing sequence $v_n\in {\mathcal N}_{\rm sc}$ (which is not empty due to Lemma \ref{le51}) for 
\[
c_2=\inf_{v\in{\mathcal N}_{\rm sc}}\J(v).
\]
 Again $\{v_n\}_n$ is bounded and, up to subsequences, converges to some $v$ weakly in $W^{s,p}_0(\Omega)$ and strongly in $L^r(\Omega)$ and similarly for $v_n^{\pm}$.
Due to $\langle \J'(v_n^\pm), v_n^\pm\rangle\leq \langle \J'(v_n), v_n^\pm \rangle=0$,   Lemma \ref{lemmalb} gives $[v_n^\pm]_{s,p}\geq \eps_0>0$ and applying   $\langle \J'(v_n), v_n^\pm \rangle=0$, \eqref{b} and Hardy's inequality, we  obtain
\[
\begin{split}
\lambda \int_\Omega|v^\pm|^r\, dx&=\lim_n \lambda \int_\Omega|v_n^\pm|^r\, dx=\lim_n\langle (-\Delta_p)^s v_n, v_n^\pm\rangle -\mu\int_\Omega\frac{|v^\pm_n|^p}{|x|^{ps}}\, dx\\
&\geq \limsup_n\, [v_n^\pm]_{s,p}^p-\mu\int_{\Omega} \frac{|v_n^{\pm}|^p}{|x|^{ps}}\, dx\geq\limsup_n \, [v_n^\pm]_{s,p}^p\left(1-\frac{\mu}{\lambda_{1, ps}}\right)\geq \left(1-\frac{\mu}{\lambda_{1, ps}}\right)\eps_0^p>0.
\end{split}
\]
Therefore $v$ is sign-changing. Next we claim that
\begin{equation}
\label{claimlsc}
u\mapsto H_\pm(u):= \langle (-\Delta_p)^s u, u^\pm\rangle -\mu \int_{\Omega}\frac{|u^\pm|^p}{|x|^{ps}}\, dx\quad \text{is sequentially weakly l.s.c.}.
\end{equation}
In order to prove it,  define, for any $u\in W^{s,p}_0(\Omega)$,
\[
|D^s_\pm u|^p(x)=\int_{\R^N} \frac{|u(x)-u(y)|^{p-2}(u(x)-u(y))(u^\pm(x)-u^\pm(y))}{|x-y|^{N+ps}}\, dy,
\]
observing that 
\[
\langle (-\Delta_p)^s u, u^\pm\rangle =\int_{\R^N}|D^s_\pm u|^p\, dx.
\]
An elementary argument based on \eqref{glk} shows that, respectively, 
\begin{equation}
\label{ppp2}
|D^s_\pm u|^p\geq |D^s u^\pm|^p.
\end{equation}
Suppose, up to subsequences, that
\[
|D^s_\pm u_n|^p\weakto^* \widetilde{\sigma}^{\pm}, \qquad |D^s u_n^\pm|^p\weakto^*\sigma^{\pm},\qquad
\frac{|u_n^\pm|^p}{|x|^{ps}}\weakto^* \nu^{\pm}.
\]
Lemma \ref{cc} ensures that for some $\nu_0^{\pm}\geq 0$
\[
\nu^{\pm}=\frac{|u^\pm|^p}{|x|^{ps}}+\nu_0^\pm\delta_0,\qquad \sigma^{\pm}\geq \lambda_{1, ps}\nu_0^{\pm}\delta_0.
\]
Inequality \eqref{ppp2}  implies that $\widetilde{\sigma}^{\pm}\geq \sigma^{\pm}$, so that 
\begin{equation}
\label{lop}
\widetilde{\sigma}^{\pm}\geq \lambda_{1, ps}\nu_0^{\pm}\delta_0.
\end{equation}
On the other hand Fatou's Lemma ensures
\[
\lim_n\int_{\R^N}|D^s_\pm u_n|^p\varphi\, dx\geq \int_{\R^N}|D^s_\pm u|^p\varphi\, dx, \qquad \forall \varphi\in C^0_c(\R^N), \ \varphi\geq 0
\]
so that 
\[
\widetilde{\sigma}^{\pm}\geq |D^s_\pm u|^p.
\]
Gathering the latter inequality with \eqref{lop} (notice that the two measures are mutually singular) we get
\[
\widetilde{\sigma}^{\pm}\geq |D^s_\pm u|^p+\lambda_{1, ps}\nu_0^\pm\delta_0.
\]
The lower semicontinuity claim \eqref{claimlsc} now follows from $\mu<\lambda_{1, ps}$, since then
\[
H_\pm(u_n)\to \widetilde{\sigma}^\pm(\R^N) -\mu\nu^\pm(\R^N)\geq H_\pm(u)+(\lambda_{1, ps}-\mu)\nu_0^\pm\geq H(u).
\]
Therefore $u\mapsto \langle \J'(u), u^\pm\rangle$ is sequentially l.s.c. as well, proving that $\langle \J'(v), v^\pm\rangle\leq 0$. This in turn implies that $v\in {\mathcal N}_{\rm sc}$ exactly as in the proof of Theorem \ref{subc} (notice that the singular term disappears due to $q=p^*_{ps}=p$). To prove that $v_n\to v$ strongly in $W^{s,p}_0(\Omega)$ we proceed as in the proof of Theorem \ref{comth}. Indeed 
\[
0=\langle \J'(v_n), v_n\rangle-\langle \J'(v), v\rangle=[v_n-v]_{s,p}^p-\mu\int_\Omega\frac{|v_n-v|^p}{|x|^{ps}}\, dx +o_n(1),
\]
so that by Hardy's inequality
\[
[v_n-v]_{s,p}^p\left(1-\frac{\mu}{\lambda_{1, ps}}\right)\leq [v_n-v]_{s,p}^p-\mu\int_\Omega\frac{|v_n-v|^p}{|x|^{ps}}\, dx =o_n(1).
\]
This shows that $\J(v)=c_2$ by continuity, and the fact that $v$ is a critical point for $\J$ follows verbatim as in the proof of Theorem \ref{comth}.
\end{proof}

\bigskip
\medskip

\end{document}